\documentclass{article}
\usepackage[T2A, T1]{fontenc}
\usepackage[main=english, russian]{babel}
\usepackage[utf8]{inputenc}
\usepackage{amsmath,amsfonts, amssymb, amsthm}
\usepackage{tikz-cd}
\usetikzlibrary{babel}
\usepackage{xcolor}
\usepackage{titlesec}
\usepackage{dutchcal}
\usepackage{marginnote}
\usepackage[a4paper,hmargin=2.5cm,vmargin=2.5cm]{geometry}
\titleformat{\subsubsection}[runin]{\normalfont\bfseries}{\thesubsubsection.}{3pt}{}
\usepackage[colorlinks=true, linktocpage=true, linkcolor={purple}, citecolor={blue}]{hyperref}

\DeclareSymbolFont{cyrillic}{T2A}{cmr}{m}{n}
\DeclareMathSymbol{\Sha}{\mathalpha}{cyrillic}{216}

\newtheorem{proposition}[subsubsection]{Proposition}
\newtheorem{theorem}[subsubsection]{Theorem}
\newtheorem{lemma}[subsubsection]{Lemma}
\newtheorem{corollary}[subsubsection]{Corollary}
\theoremstyle{definition}
\newtheorem{definition}[subsubsection]{Definition}

\newtheorem{remark}[subsubsection]{Remark}
\newtheorem{conjecture}{Conjecture}

\theoremstyle{plain}
\newtheorem{theor}{Theorem}

\renewcommand{\C}{\mathbb C}

\newcommand{\Z}{\mathbb Z}
\newcommand{\Q}{\mathbb Q}

\newcommand{\R}{\mathbb R}
\renewcommand{\O}{\mathcal O}

\newcommand{\im}{\operatorname{im}}
\renewcommand{\ker}{\operatorname{ker}}

\newcommand{\codim}{\operatorname{codim}}
\newcommand{\rk}{\operatorname{rk}}

\renewcommand{\phi}{\varphi}
\newcommand{\restrict}[1]{{\left|_{{\phantom{|}\!\!}_{#1}}\right.}}

\title{Shafarevich--Tate groups of holomorphic\\ Lagrangian fibrations II}
\author{Anna Abasheva}
\date{}

\begin{document}

\maketitle

\begin{abstract}
    Let $X$ be a compact hyperk\"ahler manifold with a Lagrangian fibration $\pi\colon X\to B$. A {\em Shafarevich--Tate twist} of $X$ is a holomorphic symplectic manifold with a Lagrangian fibration $\pi^\phi\colon X^\phi\to B$ which is isomorphic to $\pi$ locally over the base. In particular, $\pi^\phi$ has the same fibers as $\pi$. A twist $X^\phi$ corresponds to an element $\phi$ in the {\em Shafarevich--Tate group} $\Sha$ of $X$. We show that $X^\phi$ is K\"ahler when a multiple of $\phi$ lies in the connected component of unity of $\Sha$ and give a necessary condition for $X^\phi$ to be bimeromorphic to a K\"ahler manifold.
\end{abstract}

\tableofcontents



\section{Introduction}

\subsection{Definitions}
\begin{definition}
\label{definition irreducible holomorphic symplectic}
    An {\em irreducible holomorphic symplectic} manifold $X$ is a compact complex simply connected manifold admitting a closed holomorphic symplectic form $\sigma$ such that $H^0(\Omega^2_X) = \C\cdot\sigma$. If $X$ is K\"ahler, then we will call $X$ irreducible {\em hyperk\"ahler}\footnote{Most algebraic geometers use terms \textit{holomorphic symplectic manifold} and \textit{hyperk\"ahler manifolds} interchangeably. However, it is important for us to make this distinction because we will encounter non-K\"ahler holomorphic symplectic manifolds in this paper.}.
\end{definition}
\begin{definition}
\label{definition lagrangian fibration}
    A {\em Lagrangian fibration} on an irreducible holomorphic symplectic manifold $X$ is a morphism $\pi\colon X\to B$ with connected fibers to a normal variety $B$ such that the restriction of $\sigma$ to a smooth fiber is zero.
\end{definition}

If $X$ is hyperk\"ahler and the base $B$ is smooth, then $B$ is necessarily isomorphic to $\mathbb P^n$ \cite{hwang2008base}. No examples of Lagrangian fibrations on irreducible holomorphic symplectic manifolds over a base other than $\mathbb P^n$ have been discovered and conjecturally the base should always be $\mathbb P^n$.

\begin{definition}
\label{definition vertical vector fields}
Define the sheaf $T_{X/B}$ of {\em vertical vector fields} on $X$ as the kernel of the map $T_X\to \pi^*T_B/\operatorname{Tors}(\pi^*T_B)$, where $T_B := (\Omega_B)^\vee$ and $\operatorname{Tors}(\pi^*T_B)$ is the torsion subsheaf of $\pi^*T_B$.\footnote{When $B$ is smooth, the sheaf $\pi^*T_B$ is clearly locally free, hence torsion free. We do not know whether $\pi^*T_B$ is torsion free in general.}  
\end{definition} 
The flow of a vertical vector field $v$ induces a vertical automorphism $\exp(v)$ of $X$. 

\begin{definition}
\label{definition of Shafarevich-Tate group}
Consider the sheaf $Aut^0_{X/B}$ on $B$ consisting of all vertical automorphisms that are of the form $\exp(v)$ for some vertical vector field $v$ locally over $B$. The {\em Shafarevich--Tate group} of the fibration $\pi\colon X\to B$ is defined to be the group $\Sha = H^1(B,Aut^0_{X/B})$.\footnote{$\Sha$ is a letter of the Russian alphabet pronounced as {\em “Sha”}. It is the first letter in the last name \foreignlanguage{russian}{Шафаревич} (Shafarevich).}
\end{definition}

The group $\Sha$ has a beautiful geometric interpretation. Cover $B$ by open disks so that $B = \bigcup U_i$. For each subset $I$ of indices, we denote $\bigcap\limits_{i\in I} U_i$ by $U_I$ and $\pi^{-1}(U_I)$ by $X_I$. Every class $\phi\in \Sha$ can be represented by a \v{C}ech $1$-cocycle with coefficients in $Aut^0_{X/B}$. In other words, we have a vertical automorphism $\phi_{ij}$ of $X_{ij}$ for each pair of indices $i,j$, and 
\begin{equation}
\label{cocycle condition}
\phi_{ik} = \phi_{jk}\circ\phi_{ij}.    
\end{equation}
For each $i,j$ glue $X_i$ to $X_j$ by the automorphism $\phi_{ij}$ to get a new variety $X^\phi$. By the cocycle condition (\ref{cocycle condition}) the variety $X^\phi$ is a smooth Hausdorff complex manifold admitting a fibration 
$$
\pi^\phi\colon X^\phi\to B.
$$
\begin{definition}
    The manifold $X^\phi$ constructed above is called the {\em Shafarevich--Tate twist} of $X$ with respect to the class $\phi\in \Sha$. 
\end{definition}

Note that the sheaves $Aut^0_{X/B}$ and $Aut^0_{X^\phi/B}$ are isomorphic. Hence the Shafarevich--Tate group of $\pi\colon X\to B$ is the same as the Shafarevich--Tate group of $\pi^\phi\colon X^\phi\to B$. The Shafarevich--Tate twist of $X^\phi$ with respect to $\psi\in \Sha$ is isomorphic to $X^{\phi + \psi}$.

The Shafarevich--Tate group $\Sha = H^1(B,Aut^0_{X/B})$ has a structure of a topological group, possibly non-Hausdorff \cite[Subsection 3.1]{abasheva2025shafarevich}. Denote its connected component of unity by $\Sha^0$. By Theorem \ref{theorem sha naught} the group $\Sha^0$ is a quotient of $\C$ by a finitely generated subgroup. By \cite[Subsection 6.3]{abasheva2025shafarevich} the discrete part $\Sha/\Sha^0$ of $\Sha$ satisfies:
$$
    (\Sha/\Sha^0)\otimes\Q \simeq H^2(R^1\pi_*\Q).
$$
For a class $\phi\in\Sha$, we will denote by $\overline\phi$ its image in $\Sha/\Sha^0\otimes \Q$. We will denote by $\Sha'$ the set of classes $\phi\in\Sha$ such that $\overline\phi = 0$.

\begin{definition}
\label{def sht deformation}
    A {\em Shafarevich--Tate deformation} is a Shafarevich--Tate twist $X^\phi$ of $X$ with respect to an element $\phi\in\Sha^0$.
\end{definition}


\subsection{Statement of the results}

\begin{theor}[Theorem \ref{theorem deformations are kahler replacement}, Theorem \ref{theorem kahler substitute}]
\label{theorem_kahler}
    Let $\pi\colon X\to B$ be a Lagrangian fibration on an irreducible hyperk\"ahler manifold $X$. Pick a class $\phi\in\Sha'$, i.e., a class $\phi\in\Sha$ such that $r\phi$ lies in $\Sha^0$ for some positive integer $r$. Then the following holds.
    \begin{enumerate}
        \item The twist $X^\phi$ is K\"ahler.
        \item Assume $X$ is projective. Then a twist $X^\phi$ with respect to $\phi\in\Sha'$ is projective if and only if $\phi$ is torsion. 
    \end{enumerate}
\end{theor}

We proved a version of this theorem in \cite[Theorem 1.3]{abasheva2025shafarevich} for a general hyperk\"ahler manifold assuming $\phi\in \Sha^0$. The new proof does not require these assumptions. Another proof of Theorem \ref{theorem_kahler} assuming $\phi\in \Sha^0$ was given by Soldatenkov--Verbitsky in \cite{soldatenkov2024hermitian}, however, our arguments are different. 

One of the steps of the proof of Theorem \ref{theorem_kahler}(1) in the previous version of the paper, namely Lemma 4.6.1, crucially relied on the argument used by Perego in \cite[Proposition 2.6]{perego2019kahlerness} to show that certain limits of hyperk\"ahler manifolds are Fujiki class $\mathcal C$. Unfortunately, both our and Perego's arguments turned out to be faulty, as was pointed out in \cite{soldatenkov2024hermitian}. Instead of trying to fix this step of the proof, we came up with a substantially different proof of Theorem \ref{theorem_kahler} that doesn't rely at all on Perego's results. The new proof still uses some arguments from the previous version of the paper: Subsections \ref{subsection period map and torelli theorems} - \ref{subsection almost shafarevich tate family} (= Subsections 4.3 - 4.5 in the previous version) remain essentially the same. However, the new proof includes one more ingredient, which we find interesting in its own right. 
\begin{theor}[\ref{subsub proof of theorem rigidity}]
\label{theorem rigidity}
    Let $\mathcal X\to \mathbb D$ be a family of irreducible hyperk\"ahler manifolds. Suppose that the restriction $\mathcal X_U$ of $\mathcal X$ to an open subset $U\subset \mathbb D$ with $0\in \overline U$ admits a family of Lagrangian fibrations $\Pi_U\colon \mathcal X_U\to \mathcal B_U$. Assume that the central fiber $X:=X_0$ is non-projective. Then the family of fibrations $\Pi_U$ extends to a family of Lagrangian fibrations $\Pi\colon \mathcal X\to \mathcal B$ over $\mathbb D$.
\end{theor}

Note that the theorem may not hold if the central fiber is projective because in this case the isotropic line bundle $L$ inducing the Lagrangian fibration on a general fiber $X_t$ may not remain nef. As was shown in \cite{soldatenkov2025abundance}, the theorem still holds if we assume $L$ to remain nef on the central fiber. While Soldatenkov and Verbitsky use K\"ahlerness of Shafarevich--Tate deformations to prove their result on limits of Lagrangian fibrations, our proof proceeds in the opposite direction. 


\subsubsection{} A Shafarevich--Tate twist $X^\phi$ of a holomorphic symplectic manifold is holomorphic symplectic, and the fibration $\pi^\phi$ is a Lagrangian fibration \cite[Corollary 3.7]{abasheva2025shafarevich}. We can show more.

\begin{theor}[\ref{subsub proof of theorem b}]
\label{theor irr hol symp}
Let $\pi\colon X\to B$ be a Lagrangian fibration on an irreducible hyperk\"ahler manifold. Then for any $\phi\in\Sha$ we have $H^0(X^\phi, \Omega^2_{X^\phi}) = \C\cdot \sigma$, where $\sigma$ is a holomorphic symplectic form on $X^\phi$. Moreover, $H^1(X^\phi,\Q)=0$. 
\end{theor}

\subsubsection{} In the next theorem we compute the second Betti number of Shafarevich--Tate twists. Note that the differential $d_2$ on the second page of the Leray spectral sequence of $\Q_X$ for the map $\pi$ maps $H^0(B, R^2\pi_*\Q)$ to $H^2(B, R^1\pi_*\Q)\simeq (\Sha/\Sha^0)\otimes \Q$.

\begin{theor}[\ref{subsub proof of theorem c}]
\label{theorem b2}
    Let $\pi\colon X\to B$ be a Lagrangian fibration on an irreducible hyperk\"ahler manifold $X$ and $\phi\in \Sha$. Then exactly one of the following two cases occurs.
    \begin{enumerate}
        \item If the image $\overline\phi$ of $\phi$ in $H^2(R^1\pi_*\Q)$ lies in the image of $d_2$, then $b_2(X^\phi)=b_2(X)$. Moreover, there is a cohomology class $h\in H^2(X^\phi)$ which restricts to an ample class on a smooth fiber.
        \item If $\overline\phi$ is not in the image of $d_2$, then $b_2(X^\phi) = b_2(X) - 1$. In this case all cohomology classes $h\in H^2(X^\phi)$ restrict trivially to a smooth fiber.
    \end{enumerate}
\end{theor}

\begin{definition}
    A complex manifold is said to be of {\em Fujiki class $\mathcal C$} if it is bimeromorphic to a K\"ahler manifold.
\end{definition}
We will derive the following criterion for non-K\"ahlerness of Shafarevich--Tate twists as an easy corollary of Theorem \ref{theorem b2}.
\begin{theor}[\ref{subsub proof of theorem d}]
\label{theorem fujiki class c}
    Let $\pi\colon X\to B$ be a Lagrangian fibration on a hyperk\"ahler manifold $X$. Pick $\phi\in \Sha$ such that $\overline\phi$ is not in the image of $d_2$. Then $X^\phi$ is not Fujiki class $\mathcal C$, in particular, not K\"ahler.
\end{theor}

\subsubsection{Outline of the paper.} We start by recalling basic facts about Lagrangian fibrations and their Shafarevich--Tate twists in Section \ref{section preliminaries}. Many results in Section \ref{section preliminaries} were contained in our previous work \cite{abasheva2025shafarevich} but were stated assuming that the base $B$ of a Lagrangian fibration $\pi$ is smooth and $\pi$ has no multiple fibers in codimension one. We show that these assumptions are not necessary. In Section \ref{section projective twists} we will prove the second part of Theorem \ref{theorem_kahler}, which is easier than the first part. The first part of Theorem \ref{theorem_kahler} will be proven in Section \ref{sec kahler twists} after we show Theorem \ref{theorem rigidity} in Section \ref{section limits of lagrangian fibrations}. At the end of Section \ref{sec kahler twists} we will also rephrase our results in terms of the period map. In Section \ref{section topology of shafarevich} we study cohomological properties of Shafarevich--Tate twists. We will see that Shafarevich--Tate twists have trivial first cohomology in Subsection \ref{subsection fundamental group of twists} and prove that $H^0(\Omega^2_{X^\phi})$ is one-dimensional in Subsection \ref{subsection hodge numbers of twists}. These two statements immediately imply Theorem \ref{theor irr hol symp}. Finally, in Subsection \ref{subsection second cohomology of a twist} we prove Theorem \ref{theorem b2} and then show how to derive Theorem \ref{theorem fujiki class c} from Theorem \ref{theorem b2}.

\subsubsection{Acknowledgements.} I thank my advisor Giulia Sacc\`a as well as (in alphabetical order) Rodion D\'eev, Yagna Dutta, Daniel Huybrechts, Yoon-Joo Kim, Nikita Klemyatin, Morena Porzio, Evgeny Shinder, Sasha Viktorova, and Claire Voisin for their interest and helpful conversations. I'd like to thank especially Daniel Huybrechts; I benefited enormously from conversations with him while I was working on the previous version of the paper. Rodion D\'eev, Yagna Dutta, Giulia Sacc\`a, and Sasha Viktorova read a draft (or drafts) of this paper, and I truly appreciate their comments. I completed the first version of this paper during my stay in Oberwolfach, Germany at the workshop ``Algebraic Geometry: Wall Crossing and Moduli Spaces, Varieties and Derived Categories''. I am deeply grateful to the organizers of the workshop for the invitation and the opportunity to present the results of this paper. I also want to thank my friends who cheered me up while I was writing -- and especially rewriting! -- this paper: Igor, Julian, Masha, Morena Porzio, Reo, Sasha Viktorova, and Zoe. I acknowledge partial support from NSF FRG grant DMS-2052934.


\section{Preliminaries}
\label{section preliminaries}

\subsection{Lagrangian fibrations}

\subsubsection{Beauville-Bogomolov-Fujiki form.}\label{subsub bbf form}
One of the key cohomological features of hyperk\"ahler manifold is the existence of a quadratic form on their second cohomology called {\em Beauville-Bogomolov-Fujiki form (BBF form)}. 

\begin{theorem}[{\cite[Part III, Corollary 23.11 \& Proposition 23.14]{gross2003calabi}}]\label{theorem bbf}
    Let $X$ be an irreducible hyperk\"ahler manifold of dimension $2n$. Then there exists an integral symmetric non-degenerate form $q$ on $H^2(X)$ such that $\forall \alpha\in H^2(X,\Z)$,
    $$
    c_Xq(\alpha)^n = \int_X \alpha^{2n}.
    $$
    The constant $c_X$ is positive and depends only on the deformation type of $X$.
\end{theorem}
\begin{remark}
    The integral form $q$ from Theorem \ref{theorem bbf} is uniquely defined if we require it to be non-divisible.
\end{remark}
\begin{definition}
\label{definition bbf}
    The form $q$ from Theorem \ref{theorem bbf} is called the {\em Beauville--Bogomolov--Fujiki form} or {BBF form}. 
\end{definition}

\subsubsection{Fibers of Lagrangian fibrations are abelian varieties.}\label{subsub fibers are projective} Consider a Lagrangian fibration $\pi\colon X\to B$ (Definition \ref{definition lagrangian fibration}) on an irreducible holomorphic symplectic manifold $X$. A general fiber of $\pi$ is a complex torus and even an abelian variety \cite{campana2006isotrivialite}. The projectivity of smooth fibers follows easily from the theorem below.
\begin{theorem}[\cite{voisin1992stabilite, matsushita1999fibre}]
\label{theorem_restriction}
    Let $\pi\colon X\to B$ be a Lagrangian fibration on a hyperk\"ahler manifold and $F$ its smooth fiber. Then the restriction map
    $$
        H^2(X,\Q)\to H^2(F,\Q)
    $$
    has rank one.
\end{theorem}
Thanks to Theorem \ref{theorem_restriction}, for any K\"ahler class $h\in H^2(X,\R)$ some real multiple $c\cdot h$ of $h$ restricts to an integral class on $F$. The class $c\cdot h\restrict{F}$ is K\"ahler and integral, hence ample. It follows that $F$ is indeed an abelian variety.

\subsubsection{Discriminant.} The image in $B$ of singular fibers of $\pi$ is called the {\em discriminant} of the Lagrangian fibration and will be denoted by $\Delta$. It is known to be a divisor \cite[Proposition 3.1]{hwang2009characteristic}. We define $B^\circ$ to be the complement of $\Delta$ and $X^\circ:=\pi^{-1}(B^\circ)$.

\subsubsection{Vertical vector fields.}\label{subsub vertical vector fields} The holomorphic symplectic form $\sigma$ enables us to construct a lot of vertical vector fields on $X$. First, it induces an isomorphism $\Omega_X\xrightarrow{\iota_\sigma} T_X$. Let $X'$ denote the subset $\pi^{-1}(B^{reg})\subset X$. Consider the composition of maps
$$
\pi^*\Omega_{B^{reg}} \hookrightarrow \Omega_{X'}\xrightarrow[\sim]{\iota_\sigma} T_{X'} \to \pi^*T_{B^{reg}}.
$$
It is easy to see that it vanishes on $X^\circ$. Indeed, for every form $\alpha$ on an open subset of the base, the vector field dual to $\pi^*\alpha$ is tangent to smooth fibers of $\pi$. Since $\pi^*T_{B^{reg}}$ is locally free, the map $\pi^*\Omega_{B^{reg}}\to \pi^*T_{B^{reg}}$ vanishes on $X'$. Therefore, the map $\iota_\sigma$ sends $\pi^*\Omega_{B^{reg}}$ into $T_{X'/B^{reg}}$ (Definition \ref{definition vertical vector fields}). By taking pushforwards to $B^{reg}$ we obtain a map
$$
\pi_*\pi^*\Omega_{B^{reg}} \hookrightarrow \pi_*T_{X'/B^{reg}}.   
$$
Since $\pi_*\mathcal O_X\simeq\mathcal O_B$, the projection formula implies that $\pi_*\pi^*\Omega_{B^{reg}}\simeq\Omega_{B^{reg}}$, and we get a map:
\begin{equation}
\label{map omega to txb on breg}
\Omega_{B^{reg}} \hookrightarrow \pi_*T_{X'/B^{reg}}.   
\end{equation}

The sheaf $T_{X/B}$ is the kernel of the map $T_X\to \pi^*T_B/\operatorname{Tors}(\pi^*T_B)$ (Definition \ref{definition vertical vector fields}). The kernel of a map of a reflexive sheaf to a torsion-free sheaf is a reflexive sheaf, hence $T_{X/B}$ is reflexive. The pushforward of a reflexive sheaf along an equidimensional morphism is reflexive \cite[Corollary 1.7]{hartshorne1980stable}, hence $\pi_*T_{X/B}$ is reflexive as well. Therefore, the map (\ref{map omega to txb on breg}) extends to a map
\begin{equation}
\label{map omega to txb}
    \iota_\sigma\colon \Omega^{[1]}_B \hookrightarrow \pi_*T_{X/B}.
\end{equation}
Here $\Omega^{[1]}_B$ denotes the sheaf of {\em reflexive differentials} on $B$, i.e., the double dual of $\Omega_B$. Equivalently, $\Omega^{[1]}_B:= j_*\Omega_{B^{reg}}$, where $j\colon B^{reg}\hookrightarrow B$ is the embedding of the smooth locus of $B$ into $B$. Similarly, we define $\Omega^{[i]}_B$ as $j_*\Omega^i_{B^{reg}}$.

\subsubsection{} The map (\ref{map omega to txb}) turns out to be an isomorphism. We showed this fact in \cite[Lemma 2.3]{abasheva2025shafarevich} assuming that $B$ is smooth. This assumption is not necessary, as we will see very soon. The proof relies on the following elementary lemma.

\begin{lemma}
\label{lemma pullbacks of forms}
    Let $\pi\colon Y\to S$ be a proper flat morphism of possibly non-compact complex manifolds. As before, denote by $S^\circ$ the image of smooth fibers of $\pi$ and by $Y^\circ$ the preimage of $S^\circ$ in $Y$. Let $\Delta := S\setminus S^\circ$ be the discriminant locus of $\pi$. Suppose that $\alpha$ is a holomorphic $k$-form on $Y$ such that the restriction of $\alpha$ to $Y^\circ$ satisfies
    $$
    \alpha\restrict{Y^\circ} = \pi^*\beta^\circ
    $$
    for some holomorphic $k$-form $\beta^\circ$ on $S^\circ$. Then the form $\beta^\circ$ extends to a holomorphic $k$-form $\beta$ on $S$ and $\alpha = \pi^*\beta$.
\end{lemma}
\begin{proof}
    Suppose that $\alpha\restrict{\pi^{-1}(S')} = \pi^*\beta'$ for some form $\beta'$ on an open subset $S'\subset S$ with complement of codimension at least two. Then we are done. Indeed, by Hartogs theorem $\beta'$ extends to a holomorphic form $\beta$ on $S$. The forms $\pi^*\beta$ and $\alpha$ coincide on an open subset, hence they coincide on $Y$. Therefore, it is enough to prove the statement for some $S'\subset S$ as above.

    If $\codim \Delta \ge 2$, then we are done, so let us assume that $\codim \Delta = 1$. Pick a general point $b\in\Delta$. Let $U$ be a neighborhood of $b$. It is enough to show that $\beta^\circ$ extends to a holomorphic form on $U$. The fibration $\pi$ might not admit a local section in a neighborhood $U$ of $b$; yet, for some finite cover $f\colon V\to U$ ramified in $\Delta\cap U$, the base change morphism $\pi_V\colon X_V\to V$ of $\pi$ to $V$ admits a section. Call this section $s\colon V\to X_V$ and denote the map $X_V\to X$ by $F$. We obtain the following diagram
    $$
\begin{tikzcd}
X_V \arrow[d, "\pi_V"] \arrow[r, "F"]      & X_U \arrow[d, "\pi_U"] \arrow[r, hook] & X \arrow[d, "\pi"] \\
V \arrow[r, "f"] \arrow[u, "s", bend left] & U \arrow[r, hook]                      & B                 
\end{tikzcd}
    $$
    The following equality of forms on $F^{-1}((X_U)^\circ)$ holds:
    $$
    F^*\alpha\restrict{F^{-1}((X_U)^\circ)} = \pi_V^*f^*\beta^\circ.
    $$
    It follows that the form $s^*F^*\alpha$ coincides with $f^*\beta^\circ$ on $V^\circ$. Therefore $f^*\beta^\circ$ can be extended to a form $\beta_V:=s^*F^*\alpha$ on $V$. As we will see in a moment, this implies that $\beta^\circ$ extends to a holomorphic form on $U$. Indeed, choose coordinates $(t,z_1,\ldots z_{n-1})$ on $U$ and $(s,z_1,\ldots z_{n-1})$ on $V$ such that $\Delta\cap U = \{t=0\}$ and the map $f$ sends $(s,z_1,\ldots z_{n-1})$ to $(s^k, z_1,\ldots z_{n-1})$. Write 
    $$
    \beta^\circ = hdt + \sum\limits_{i=1}^{n-1}h_i dz_i
    $$
    for some functions $h$ and $h_i$ on $U^\circ$. Then
    $$
    f^*\beta^\circ = kh(s^k,z)s^{k-1}ds + \sum\limits_{i=1}^{n-1} h_i dz_i.
    $$
    The form $f^*\beta^\circ$ extends to a holomorphic form on $V$. Hence the functions $h_i$'s extend to holomorphic functions on $V$. They are bounded on $V$, hence bounded on $U$. Therefore, $h_i$'s extend to holomorphic functions on $U$. The function $h(s^k,z)s^{k-1}$ is also bounded, hence so is
    $$
    h(t,z)t = h(s^k,z)s^k.
    $$
    Therefore, $h$ has at worst a simple pole at $\Delta$. But the form
    $$
    f^*\frac{dt}{t} = k\frac{ds}{s},
    $$
    is not holomorphic. Hence $h$ is actually holomorphic on $U$. It follows that $\beta^\circ$ extends to a holomorphic form on $U$.
\end{proof}

\begin{theorem}
\label{theorem iso omega and t}
    The map $\iota_\sigma\colon\Omega^{[1]}_B\to \pi_*T_{X/B}$ is an isomorphism.
\end{theorem}
\begin{proof}
    This map is definitely an isomorphism over $B^\circ$ and is injective (\ref{subsub vertical vector fields}). It is enough to show that it is surjective. Let $v$ be a vertical vector field over an open subset $U\subset B$. Then the form $\iota_v\sigma$ equals $\pi^*\beta^\circ$ for some holomorphic $1$-form $\beta^\circ$ on $B^\circ\cap U$. By Lemma \ref{lemma pullbacks of forms}, the form $\beta^\circ$ extends to a holomorphic form $\beta$ on $U^{reg}$ and $\iota_v\sigma\restrict{\pi^{-1}(U^{reg})}$ coincides with $\pi^*\beta$. Hence the map $\iota_\sigma$ sends the form $\beta$, considered as a section of $\Omega^{[1]}_B$ over $U$, to $v$. 
\end{proof}

\subsubsection{Higher pushforwards of $\mathcal O_X$}\label{subsub higher pushforwards} When the base $B$ of a Lagrangian fibration is smooth, the higher pushforward sheaves $R^i\pi_*\mathcal O_X$ are locally free \cite{matsushita2005higher}. Without the smoothness assumption one can show that the sheaves $R^i\pi_*\mathcal O_X$ are reflexive for all $i\ge 0$ \cite[Proposition 3.6]{ou2019lagrangian}. Let $\omega$ be a K\"ahler form on $X$. Consider the composition of maps
$$
\Omega^{[1]}_B\xrightarrow{\iota_\sigma} \pi_*T_{X/B}\xrightarrow{f_\omega} R^1\pi_*\mathcal O_X. 
$$
Here $f_\omega$ sends a vertical vector field $v$ to the class $[\iota_v\omega]_{\bar\partial}$ of the $\bar\partial$-closed $(0,1)$-form $\iota_v\omega$ under the $\bar\partial$-differential. 


\begin{theorem}[\cite{ou2019lagrangian},\cite{matsushita2005higher}]
\label{theorem matsushita}
    Let $\pi\colon X\to B$ be a Lagrangian fibration on a projective manifold. Then the map $\Omega^{[1]}_B\to R^1\pi_*\mathcal O_X$ and the induced maps $\Omega^{[i]}_B\to R^i\pi_*\mathcal O_X$ are isomorphisms.
\end{theorem}

\begin{corollary}
\label{corollary matsushita}
    Let $\pi\colon X\to B$ be a Lagrangian fibration on an irreducible hyperk\"ahler manifold, not necessarily projective. Then Theorem \ref{theorem matsushita} holds for any Shafarevich--Tate twist $X^\phi$ of $X$, in particular for $X$ itself, i.e.,
    $$
    R^i\pi^\phi_*\mathcal O_{X^\phi}\simeq \Omega^{[i]}_B.
    $$
\end{corollary}
\begin{proof}
    By \cite[Theorem 3.5]{huybrechts1999compact} any non-trivial family of deformations of an irreducible hyperk\"ahler manifolds contains a projective deformation. Therefore there exists a projective Shafarevich--Tate deformation $\pi^\psi\colon X^\psi \to B$ of the Lagrangian fibration $\pi\colon X\to B$. It follows from Theorem \ref{theorem matsushita} that
    $$
    R^i\pi^\psi_*\mathcal O_{X^\psi}\simeq \Omega^{[i]}_B.
    $$
    The sheaf of groups $Aut^0_{X/B}$ acts trivially on $R^i\pi_*\mathcal O_X$. Indeed, the restriction of $R^i\pi_*\mathcal O_X$ to $B^\circ$ is a vector bundle with fibers $H^{0,i}(F)$. Automorphisms in $Aut^0_{X/B}$ act trivially on $H^{0,i}(F)$ for any smooth fiber $F$. Thus the action of $Aut^0_{X/B}$ on $R^i\pi_*\mathcal O_X$ is trivial over $B^\circ$. The sheaf $R^i\pi_*\mathcal O_X$ is torsion-free, hence the action of $Aut^0_{X/B}$ is trivial everywhere. 
    
    We obtain that for any $\phi\in\Sha$
    $$
    R^i\pi^\phi_*\mathcal O_{X^\phi} \simeq R^i\pi^\psi_*\mathcal O_{X^\psi}\simeq \Omega^{[i]}_B.
    $$
\end{proof}

\begin{remark} 
\label{remark h0 is one}
It follows from Corollary \ref{corollary matsushita} that the sheaves $R^i\pi_*\mathcal O_X$ are locally free on $B^{reg}$. The base change theorem \cite[Chapter 5, Corollary 2\&3]{mumford1974abelian} implies that for all points $b\in B^{reg}$ the dimension of $H^i(\mathcal O_{\pi^{-1}(b)})$ does not depend on $b$. In particular, $h^0(\mathcal O_{\pi^{-1}(b)}) = 1$ for every $b\in B^{reg}$.
\end{remark}


\begin{theorem}
\label{theorem cohomology of omega}
    Let $B$ be the base of a Lagrangian fibration on an irreducible hyperk\"ahler manifold $X$. Then the cohomology groups $H^j(B, \Omega^{[i]}_B)$ are the same as for $B=\mathbb P^n$.
\end{theorem}

\begin{proof}
{\bf Step 1.} By Corollary \ref{corollary matsushita}, $H^j(\Omega^{[i]}_B)\simeq H^j(R^i\pi_*\mathcal O_X)$. It follows from a result by Koll\'ar \cite[p.172]{kollar1986higher} that
$$
 R\pi_*\mathcal O_X \simeq \bigoplus R^i\pi_*\mathcal O_X[-i].   
$$
Therefore the Leray spectral sequence for $\mathcal O_X$ degenerates on $E^2$ and
$$
h^{0,k}(X) = \sum\limits_{i=0}^k h^{k-i}(R^i\pi_*\mathcal O_X).
$$
When $k$ is odd $h^{0,k} = 0$ and when $k$ is even $h^{0,k} = 1$. We see immediately that $H^j(R^i\pi_*\mathcal O_X) = 0$ when $i+j$ is odd. When $k$ is even, there is exactly one $i\le k$ such that $H^{k-i}(R^i\pi_*\mathcal O_X)$ is non-zero.

\hfill

{\bf Step 2.} We will show that $H^i(R^i\pi_*\mathcal O_X)$ does not vanish. This will complete the proof. Consider the filtration $F^iH^{0,k}(X)$ on $H^{0,k}(X)$ induced by the Leray spectral sequence. First, consider the case $k=2$. The cohomology group $H^{0,2}(X)$ is generated by $\overline\sigma$. The restriction of $\overline\sigma$ to a smooth fiber is zero, hence the image of $\overline\sigma$ in $H^0(B, R^2\pi_*\mathcal O_X)$ vanishes. The form $\overline\sigma$ is non-degenerate, hence not the pullback of a $(0,2)$-form on the base even locally. Therefore $F^0H^{0,2}(X) = 0$ and $F^1H^{0,2}(X) = F^2H^{0,2}(X) = H^{0,2}(X)$. 

It follows that $\overline\sigma^i\in F^iH^{0,2i}(X)$ for all $i$. Suppose that we know that $\overline\sigma^i\not\in F^{i-1}H^{0,2i}(X)$. Then $H^i(R^i\pi_*\mathcal O_X) = F^iH^{0,2i}(X)/F^{i-1}H^{0,2i}(X)$ is non-zero, and we are done. If $\overline\sigma^i$ happens to be contained in  $F^{i-1}H^{0,2i}(X)$, then $\overline\sigma^n$ is contained in $F^{n-1}H^{0,2n}(X)$. However, $F^{n-1}H^{0,2n}(X)$ vanishes for dimension reasons. Indeed, $H^{n+k}(R^{n-k}\pi_*\mathcal O_X) = 0$ for $k>0$. Hence $\overline\sigma^n = 0$, contradiction.
\end{proof}

\begin{remark}
    A base of a Lagrangian fibration behaves like $\mathbb P^n$ from many points of view (conjecturally because it is always $\mathbb P^n$). We encourage an interested reader to look into the wonderful survey \cite{huybrechts2022lagrangian} for details.
\end{remark}

\subsection{Shafarevich--Tate group}
\label{subsection Shafarevich--Tate group}

\subsubsection{Structure of the Shafarevich--Tate group.}

Recall that the sheaf of groups $Aut^0_{X/B}$ is defined as the image of the exponential map $\pi_*T_{X/B}\to Aut_{X/B}$ (Definition \ref{definition of Shafarevich-Tate group}). Define $\Gamma$ to be the kernel of this map. The short exact sequence
$$
    0\to \Gamma\to \pi_*T_{X/B}\to Aut^0_{X/B}\to 0
$$
induces the long exact sequence of cohomology groups:
\begin{equation}
\label{seq sha}
    H^1(\Gamma)\to H^1(\pi_*T_{X/B})\to \Sha \to H^2(\Gamma).
\end{equation}
We will call the image of $H^1(\pi_*T_{X/B})$ in $\Sha$ the {\em connected component of unity} of $\Sha$ and will denote it by $\Sha^0$. The quotient $\Sha/\Sha^0$ is the {\em discrete part} of $\Sha$.

The sequence (\ref{seq sha}) is exact on the right. Indeed, the cohomology group $H^2(\pi_*T_{X/B})$ is isomorphic to $H^2(\Omega^{[1]}_B)$ by Theorem \ref{theorem iso omega and t}. By Theorem \ref{theorem cohomology of omega} this cohomology group vanishes. Similarly, the vector space $H^1(\pi_*T_{X/B})$ is isomorphic to $H^1(B,\Omega^{[1]}_B)$ and is one-dimensional.

\subsubsection{Degenerate twistor deformations.}\label{subsub degenerate twistor deformations}There is a useful differential geometric point of view on Shafarevich--Tate deformations  \cite[Subsection 2.3]{abasheva2025shafarevich}. Let $\sigma$ be a holomorphic symplectic form on $X$ and $\alpha$ be a closed $(1,1)$-form on $B$. The form $\sigma + t\pi^*\alpha$ is obviously not holomorphic, but it turns out that there exists a {\em different} complex structure $I_t$ on $X$ making $\sigma+t\pi^*\alpha$ holomorphic symplectic \cite[Section 2.2]{soldatenkov2024hermitian}. Moreover, such a complex structure is unique.

\begin{definition}
\label{definition deg tw}
    Denote by $X_t$ the manifold $X$ with the new complex structure $I_t$. It is called a {\em degenerate twistor deformation} of $X$.
\end{definition} 
It is not hard to see that the fibration $\pi\colon X_t\to B$ is holomorphic and Lagrangian with respect to the new complex structure.

Degenerate twistor deformations form a family
$$
\Pi\colon \mathcal X \to \mathbb A^1,
$$
and the fiber of $\Pi$ over $t\in \mathbb A^1$ is isomorphic to the degenerate twistor deformation $X_t$.
\begin{definition}[{\cite[Definition 2.14, Definition 3.4]{abasheva2025shafarevich}}]
\label{definition sht family}
    The family $\Pi\colon \mathcal X\to \mathbb A^1$ is called the {\em degenerate twistor family} or the {\em Shafarevich--Tate family}.
\end{definition}
We will see in Theorem \ref{theorem follows word by word} that all degenerate twistor deformations are Shafarevich--Tate deformations (Definition \ref{def sht deformation}). That justifies the use of the term Shafarevich--Tate family.

\subsubsection{The connected component of unity of $\Sha$.}\label{subsub gamma and r1} The isomorphism $f_\omega\colon \pi_*T_{X/B}\to R^1\pi_*\mathcal O_X$ from \ref{subsub higher pushforwards} sends the subsheaf $\Gamma\subset \pi_*T_{X/B}$ into $R^1\pi_*\Q$ \cite[Proposition 4.4]{abasheva2025shafarevich}. In the same paper we showed that the sheaf $\Gamma_\Q:=\Gamma\otimes\Q$ is isomorphic to $R^1\pi_*\Q$. The exact sequence (\ref{seq sha}) implies that 
$$
\Sha^0 = H^1(B,\pi_*T_{X/B})/\im H^1(B,\Gamma).
$$
The isomorphism $f_\omega\colon  \pi_*T_{X/B}\to R^1\pi_*\mathcal O_X$ identifies $\Sha^0$ with a quotient of 
\begin{equation}
\label{eq sha as h1r1pistarox}
H^1(B,R^1\pi_*\mathcal O_X)/\im H^1(R^1\pi_*\Z)
\end{equation}
by a finite subgroup. In Theorem \ref{theorem sha naught} we will describe $\Sha^0$ in terms of cohomology of $X$. First, let us introduce some notation. Let $W_\Z\subset H^2(X,\Z)$ be the subgroup of cohomology classes on $X$ that restrict trivially to all fibers. By \cite{matsushita1999fibre} $\operatorname{Pic}(B)$ has rank one. Denote by $\eta$ the class of the pullback of the ample generator of $\operatorname{Pic}(B)/\operatorname{Tors}(\operatorname{Pic}(B))$ to $X$. 

\begin{definition}
\label{definition isogenous}
    Let $G_i$, $i=1,2$ be two abelian groups of the form $G_i = \mathbb C^k/\Lambda_i$, where $\Lambda_i$ is a finitely generated subgroups of $\mathbb C^k$. We will call $G_1$ and $G_2$ {\em isogenous} if the subgroup $\Lambda_1\cap\Lambda_2$ is of finite index in both $\Lambda_1$ and $\Lambda_2$. Equivalently, the subspace $\Lambda_1\otimes\Q\subset\mathbb C^k$ coincides with $\Lambda_2\otimes\Q$
\end{definition}

\begin{theorem}
\label{theorem sha naught}
    Let $\pi\colon X\to B$ be a Lagrangian fibration on an irreducible hyperk\"ahler manifold $X$. Then the group $\Sha^0$ is isogenous to 
    $$
    H^{0,2}(X)/p(H^2(X,\mathbb Z)),
    $$
    where $p\colon H^2(X,\mathbb Z)\to H^{0,2}(X)$ is the Hodge projection.
\end{theorem}
\begin{proof}
    By \cite[Proposition 4.7]{abasheva2025shafarevich}, the Leray spectral sequence induces the following isomorphisms:
    $$
    H^1(B,R^1\pi_*\mathcal O_X) \simeq H^{0,2}(X), \:\:\:\:\:\text{and}\:\:\:\:\: H^1(B, R^1\pi_*\Z) = W_\Z/\eta.
    $$
    It follows from (\ref{eq sha as h1r1pistarox}) that $\Sha^0$ is isogenous to 
    $$
    H^{0,2}(X)/p(W_\Z).
    $$
    For every ring $\mathcal R$ define $W_{\mathcal R}:= W_\Z\otimes\mathcal R$. It is enough to show that $p(W_\Q) = p(H^2(X,\Q))$. The inclusion $p(W_\Q)\subset p(H^2(X,\Q))$ is clear. For the opposite inclusion, note that $W_\C$ contains $\sigma$ and $\overline\sigma$ \cite[Lemma 3.5]{abasheva2025shafarevich}. Therefore $(W_\Q)^\perp$ is contained in $H^{1,1}(X)$. It is a rational subspace, hence $(W_\Q)^\perp\subset NS_\Q(X)$. It follows that
    $$
    T_\Q(X):=NS_\Q(X)^\perp\subset W_\Q.
    $$
    The image of $T_\Q(X)$ under the Hodge projection coincides with the image of $H^2(X,\Q)$. Indeed, the kernel of $p\colon H^2(X,\Q)\to H^{0,2}(X)$ is $NS_\Q(X)$. Therefore,
    $$
    p(H^2(X,\Q)) = p(T_\Q(X))\subset p(W_\Q),
    $$
    and we are done.
\end{proof}
As an immediate corollary we obtain:
\begin{corollary}
\label{corollary dense}
    The set of torsion elements of $\Sha^0$ is dense in $\Sha^0$. 
\end{corollary}
\begin{proof}
    By Theorem \ref{theorem sha naught} it is enough to prove the same statement for the group $H^{0,2}(X)/p(H^2(X,\Z))$. The subgroup of torsion elements of this group is $p(H^2(X,\Q))/p(H^2(X,\Z))$. The projection $H^2(X,\R)\to H^{0,2}(X)$ is surjective and $H^2(X,\Q)$ is dense in $H^2(X,\R)$, hence the claim.
\end{proof}

\subsubsection{Degenerate twistor deformations are Shafarevich--Tate twists.} By Theorem \ref{theorem iso omega and t} and Corollary \ref{corollary matsushita} the following one-dimensional vector spaces are isomorphic 
\begin{equation}
\label{eq isos isos}
H^1(\pi_*T_{X/B})\simeq H^1(\Omega^{[1]}_B) \simeq H^{1,1}(R^1\pi_*\mathcal O_X)\simeq H^{0,2}(X) \simeq \C.
\end{equation}
Let $\sigma$ be a holomorphic symplectic form on $X$. Pick a $d$-closed $(1,1)$-form $\alpha$ on $B$, whose class in $H^1(\Omega^{[1]}_B)$ is non-trivial. We may and will choose the isomorphisms (\ref{eq isos isos}) in such a way that $[\alpha]\in H^1(\Omega^{[1]}_B)$ is identified with $\overline\sigma\in H^{0,2}(X)$, which is identified with $1\in\C$. 
\begin{theorem}
\label{theorem follows word by word}
     Let $\pi\colon X\to B$ be a Lagrangian fibration on a hyperk\"ahler manifold. For every $t\in H^1(\pi_*T_{X/B})\simeq \C$ consider its image $\phi_t\in \Sha$ by the map (\ref{seq sha}). Then the degenerate twistor deformation $X_t$ is isomorphic to the Shafarevich--Tate twist $X^{\phi_t}$ of $X$ by $\phi_t$. This isomorphism preserves the Lagrangian fibrations.
\end{theorem}
\begin{proof}
In \cite[Theorem 3.8 = Theorem 1.2]{abasheva2025shafarevich} this result was proven under the additional assumptions that $B$ is smooth and $\pi$ has no multiple fibers in codimension one. The proof actually does not use these assumptions. The reader can mentally replace $\Omega^1_B$ in the proof of \cite[Theorem 3.8]{abasheva2025shafarevich} with $\Omega^{[1]}_B$ and keep in mind that thanks to Theorem \ref{theorem cohomology of omega}
$$
H^1(\Omega^{[1]}_B) \simeq H^1(\pi_*T_{X/B})\simeq \C
$$
regardless of whether the base is smooth or fibers in codimension one are non-multiple.
\end{proof}

\subsubsection{The discrete part of $\Sha$.} The isomorphism $\Gamma_\Q\simeq R^1\pi_*\Q$ gives an easy description of the discrete part of $\Sha$. By the exact sequence (\ref{seq sha}), the discrete part $\Sha/\Sha^0$ satisfies
$$
(\Sha/\Sha^0)\otimes \Q \simeq H^2(\Gamma_\Q) \simeq H^2(R^1\pi_*\Q).
$$



\section{Projective twists}
\label{section projective twists}



\subsection{N\'eron--Severi group of Shafarevich--Tate twists}
\label{subsection neron-severi}

The goal of this chapter is to prove the second part of Theorem \ref{theorem_kahler}. It will follow from the statement below:

\begin{theorem}
\label{theorem iso between neron severis}
    Let $\pi\colon X\to B$ be a Lagrangian fibration on an irreducible holomorphic symplectic manifold, and $\phi\in\Sha$ a torsion element. Then there is a natural isomorphism
    \begin{equation}
    \label{iso between neron severis}
    NS_\Q(X)/\eta \to NS_\Q(X^\phi)/\eta,
    \end{equation}
    where $\eta$ is the pullback of an ample class on $B$. Moreover, the isomorphism (\ref{iso between neron severis}) sends 
    \begin{itemize}
        \item classes on $X$ with cohomologically trivial restriction to smooth fibers to classes with cohomologically trivial restriction to smooth fibers;
        \item relatively ample classes to relatively ample classes.
    \end{itemize} 
\end{theorem}


The proof of Theorem \ref{theorem iso between neron severis} relies on the following idea. Pick a line bundle $L$ on $X$ and cover $B$ by open disks $U_i$. Let $L_i$ denote the restriction of $L_i$ to $X_i$. We will see that for some $s\in\Z_{>0}$, the line bundles $L_i^s$ can be glued into a line bundle on $X^\phi$. We will use the following lemmas in our gluing construction.

\begin{lemma}
\label{lemma torsion in sha}
Let $\phi$ be an $r$-torsion element in $\Sha$. Cover $B$ by small open subsets $U_i$ and represent $\phi$ by a \v{C}ech cocycle $(\phi_{ij})$, $\phi_{ij}\in Aut^0_{X/B}(U_{ij})$. Then we can choose $\phi_{ij}$ in such a way that $r\phi_{ij}$ is the identity automorphism of $X_{ij}$ for each $i,j$.
\end{lemma}
\begin{proof}
    Since the class of $r\phi$ is trivial in $\Sha$, we can find automorphisms $\beta_i\in Aut^0_{X/B}(U_i)$ such that
    $$
    r\phi_{ij} = \beta_j - \beta_i.
    $$
    There exist automorphisms $\gamma_i$ such that $r\gamma_i = \beta_i$. Indeed, we can write $\beta_i = \exp(v_i)$ for some vertical vector field $v_i$. The automorphism $\gamma_i:=\exp(v_i/r)$ will do the job. Replace $\phi_{ij}$ with $\phi_{ij} + \gamma_i - \gamma_j$. The new set of automorphisms satisfies the condition of the lemma.
\end{proof}

\begin{lemma}
\label{lemma theorem of the square}
    Let $L$ be a line bundle on an abelian variety $A$ and $t$ an $r$-torsion element of $A$. Then 
    $$
    t^*L^r\simeq L^r.
    $$
\end{lemma}
\begin{proof}
    Consider the morphism $\phi_L\colon A\to A^\vee$ sending $x$ to $x^*L\otimes L^{-1}$. The map $\phi_L$ is a homomorphism because any morphism of abelian varieties sending zero to zero is a homomorphism \cite[subsection 4, Corollary 1]{mumford1974abelian}. Therefore $\phi_L(t)$ is an $r$-torsion line bundle, i.e.,
    $$
    (\phi_L(t))^r = t^*L^r\otimes L^{-r}\simeq \mathcal O_A.
    $$
\end{proof}

\begin{lemma}
\label{lemma raynaud}
    Let $\pi\colon Y\to S$ be a proper flat morphism between normal varieties such that $h^0(\mathcal O_{Y_b})=1$ for all $b\in S$ outside a codimension at least two subset of $S$ and $S$ is locally $\mathbb Q$-factorial. Consider a line bundle $M$ on $Y$ with the following properties:
    \begin{enumerate}
        \item the restriction of $M$ to any smooth fiber is trivial;
        \item the restriction of $M$ to any fiber $Y_b$ lies in $Pic^0(Y_b)$. Here $Pic^0(Y_b)$ is the connected component of unity of $Pic(Y_b)$.
    \end{enumerate}
    Then some positive multiple $M^s$ of $M$ for $s\in\Z_{>0}$ is isomorphic to the pullback of a line bundle from $S$.  
\end{lemma}
\begin{proof}
{\bf Step 1.} It is enough to show this statement for some $S'\subset S$ with complement of codimension at least two. Indeed, suppose that $M^s|_{\pi^{-1}(S')}$ is isomorphic to $\pi^*K'$ for a line bundle $K'$ on $S'$. We can extend $K'$ to a line bundle $K$ on $S$ because $S$ is $\mathbb Q$-factorial. The line bundles $\pi^*K$ and $M^s$ are isomorphic outside a codimension at least two subset of $Y$, hence they are isomorphic.

\hfill

{\bf Step 2.}  Denote by $\Delta$ the discriminant locus of $\pi$. If $\codim \Delta \ge 2$, then we are done thanks to Step 1. So we may assume $\codim \Delta = 1$. Consider the group $E_b\subset Pic(Y_b)$ of line bundles $L$ on $Y_b$ with the following property: there exists a line bundle $\tilde L$ on $Y$ which is trivial on smooth fibers and restricts to $L$ on $Y_b$. By Raynaud's theorem \cite[Introduction]{raynaud1970specialisation}, $E_b$ has dimension $h^0(\mathcal O_{Y_b})-1$ for a general point $b\in \Delta$. The assumption that $h^0(\mathcal O_{Y_b})=1$ for a general point $b\in\Delta$ implies that $E_b$ is discrete for any fiber of $\pi$ over a general point $b\in\Delta$. The line bundle $M_b:=M\restrict{Y_b}$ is in $E_b$ by the first property. By the second assumption, $M_b\in Pic^0(Y_b)$. Consider the group $\langle M_b\rangle$ generated by $M_b$ inside $Pic^0(Y_b)$. It is contained inside $E_b$, hence is discrete. Since the group space $Pic^0(Y_b)$ is of finite type \cite[Proposition 9.5.3]{fantechi2005fundamental}, the group $\langle M_b\rangle$ is of finite type as well. Hence $\langle M_b\rangle$ is finite, in other words, $M_b$ is torsion. Therefore, some power $M^s$ of $M$ restricts trivially to all fibers over $S'\subset S$ with complement of codimension at least two. Define a line bundle $K':=\pi_*M^s\restrict{\pi^{-1}(S')}$. The natural map $\pi^*K'\to M^s\restrict{\pi^{-1}(S')}$ is an isomorphism.
\end{proof}

\subsubsection{Proof of Theorem \ref{theorem iso between neron severis}.}
\begin{proof}
    {\bf Step 1.} Pick a line bundle $L$ on $X$. As before, choose a \v{C}ech cocycle $(\phi_{ij})$ with $r\phi_{ij} = 0$ representing an $r$-torsion class $\phi\in \Sha$. We will construct an isomorphism:
    $$
    f_{ij}\colon\phi_{ij}^*L^s\restrict{\pi^{-1}(U_{ij)}}\to L^s\restrict{\pi^{-1}(U_{ij)}}.
    $$
    for some $s\in\Z_{>0}$. The line bundle $\phi_{ij}^*L_j^r\otimes L_i^{-r}$ on $X_{ij}$ restricts trivially to smooth fibers by Lemma \ref{lemma theorem of the square}. Moreover, it satisfies the second condition of Lemma \ref{lemma raynaud} because $\phi_{ij}\in Aut^0_{X/B}$. For every $b\in B^{reg}$ the fibers $\pi^{-1}(b)$ satisfy $h^0(\mathcal O_{\pi^{-1}(b)}) = 1$, see \ref{subsub higher pushforwards}. By Lemma \ref{lemma raynaud} some multiple of $\phi_{ij}^*L^r\otimes L^{-r}$ is the pullback of a line bundle on $U_{ij}$. When the subsets $U_i$'s are sufficiently small, all line bundles on $U_{ij}$'s are trivial. Therefore the sheaves $\phi_{ij}^*L^s\restrict{pi^{-1}(U_{ij})}$ and $L^s\restrict{\pi^{-1}(U_{ij)}}$ are isomorphic.

    \hfill
    
    {\bf Step 2.} The isomorphisms $f_{ij}$ might not a priori satisfy the cocycle condition. In other words, the following map
    $$
    f_{ij}^{-1}\circ\phi_{ij}^*f_{jk}^{-1}\circ f_{ik}
    $$
    is some automorphism of $L\restrict{U_{ijk}}$, which might not be trivial. Denote it by $\lambda_{ijk}$. The automorphism $\lambda_{ijk}$ is a multiplication by a non-zero holomorphic function on $X_{ijk}$, which must be the pullback of a function on the base. Therefore the automorphisms $\lambda_{ijk}$ define a \v{C}ech 2-cocycle on $B$ with coefficients in $\mathcal O_B^\times$. 
    
    Consider the following chunk of the long exact sequence of cohomology of the exponential exact sequence on $B$:
    $$
    H^2(B,\mathcal O_B)\to H^2(B,\mathcal O_B^\times)\to H^3(B,\Z)\to H^3(B,\mathcal O_B).
    $$
    By Theorem \ref{theorem cohomology of omega}, the cohomology groups $H^2(B,\mathcal O_B)$ and $H^3(B,\mathcal O_B)$ vanish. Hence $H^2(B,\mathcal O_B^\times)$ is isomorphic to $H^3(B,\Z)$. The cohomology groups $H^i(B,\Q)$ are the same as for $\mathbb P^n$ \cite[Theorem 0.2]{shen2022topology}, in particular $H^3(B,\Z)$ is torsion. Hence some power, say $s'$, of the cocycle $(\lambda_{ijk})$ vanishes. Replace the line bundle $L^s$ with $L^{ss'}$ and the isomorphisms $f_{ij}$ with $f_{ij}^{\otimes s'}$. Then $\lambda_{ijk}$ gets replaced with $\lambda_{ijk}^{s'}$, which is a coboundary. Write $\lambda_{ijk}^{s'} = \mu_{ij}\mu_{jk}\mu_{ki}$ for some $\mu_{ij}\in\mathcal O_B^\times(U_{ij})$. Then the isomorphisms $(\mu_{ij}^{-1}\cdot f_{ij})$ satisfy the cocycle condition. It follows that we can glue the line bundles $L^s\restrict{\pi^{-1}(U_{ij})}$ into a global line bundle $L^\phi$ on $X^\phi$.
    
    The line bundle $L^\phi$ depends only on the choice of $\mu_{ij}\in \mathcal O_B^\times(U_{ij})$. Different choices of $\mu_{ij}$ differ by a $1$-cocycle with coefficients in $\mathcal O_B^\times(U_{ij})$. Therefore, $L^\phi$ is well-defined up to the pullback of a line bundle on $B$. We construct a map
    $$
    NS_\Q(X)/\eta \to NS_\Q(X^\phi)/\eta
    $$
    by sending the class of $L$ in $NS_\Q(X)/\eta$ to the class $[L^\phi]/(ss')\in NS_\Q(X^\phi)/\eta$.

    \hfill

    {\bf Step 3.} The restriction of $L^\phi$ to $X^\phi_i$ coincides with a power of $L_i$. Therefore the class of $L^\phi$ in $NS_\Q(X)$ has trivial restriction to smooth fibers if and only if this is true for $L$, and $L^\phi$ is relatively ample if and only if so is $L$.
\end{proof}


\subsection{Criterion of projectivity}\label{subsection algebraic points}

In this Subsection we will prove Theorem \ref{theorem_kahler}(2). We will begin with a useful general criterion of projectivity of Lagrangian fibrations, or more generally, of irreducible hyperk\"ahler manifold with an integral isotropic class.

\begin{lemma}
\label{lemma eta perp}
    Let $X$ be an irreducible hyperk\"ahler manifold with an isotropic class $\eta\in H^2(X,\Z)$. Then $X$ is projective if and only if $NS(X)\not\subset \eta^\perp$. In particular, if $X$ is not projective, then all curves on $X$ are orthogonal to $\eta$.
\end{lemma}
\begin{proof}
    If $X$ is projective and $\alpha$ is an ample class on $X$, then $q(\alpha, \eta)\ne 0$. Indeed, the BBF form is negative definite on $\alpha^\perp\in H^{1,1}(X)$ \cite[Corollary 23.11]{gross2003calabi}, hence $\alpha$ can't be orthogonal to an isotropic class. Conversely, let $\alpha\in H^2(X,\Z)$ be a class such that $q(\alpha,\eta)\ne 0$. By possibly changing the sign of $\alpha$, we may assume that $q(\alpha,\eta) > 0$. Then the class $\alpha':= \alpha + t\eta$ has positive square for a large enough $t\in\Z$. Indeed,
    $$
    q(\alpha + t\eta) = q(\alpha) + 2tq(\alpha,\eta).
    $$
    By \cite[Thm. 3.11]{huybrechts1999compact}, an irreducible hyperk\"ahler manifold with a positive integral class $\alpha'\in H^2(X,\Z)$ is projective.

\end{proof}

Denote by $p \colon H^2(X, \C) \to H^{0,2}(X)$ the Hodge projection. Let $\sigma$ be the holomorphic symplectic form such that $q(\sigma,\overline\sigma) = 1$. In the following lemma we will give a characterization of torsion in $\Sha^0$ in terms of the BBF form.

\begin{lemma}
\label{torsion in sha}
Let $\pi\colon X\to B$ be a Lagrangian fibration. Consider a class $t\overline\sigma \in H^{0,2}(X)$, $t\in \C$. Let $[t\overline\sigma]$ denote its image in $\Sha^0$. Then $[t\overline\sigma]$ is torsion if and only if there exists a rational class $l\in H^2(X,\mathbb Q)$ such that $q(l,\sigma) = t$.
\end{lemma}

\begin{proof}
By Theorem \ref{theorem sha naught}, the group $\Sha^0$ is isogenous to $H^{0,2}(X)/p(H^2(X,\Z)$. Hence, the class $[t\overline\sigma]$ is torsion if and only if $t\overline\sigma$ lies in the image of the group $H^2(X,\Q)$. The Hodge projection $H^2(X,\mathbb C)\to H^{0,2}(X)$ is given by the formula
$$
v\mapsto q(v,\sigma)\overline\sigma.
$$
Hence, $t\overline\sigma$ lies in the image of $H^2(X,\Q)$ under the Hodge projection if and only if there exists a class $l\in H^2(X,\Q)$ such that $q(l,\sigma)=t$.
\end{proof}

\hfill

Instead of proving Theorem \ref{theorem_kahler}(2) directly, we will show a more general statement.

\begin{theorem}
\label{theorem kahler substitute}
    Let $\pi\colon X\to B$ be a Lagrangian fibration on a projective hyperk\"ahler manifold, and $\phi\in\Sha'$. Then the following are equivalent:
    \begin{enumerate}
        \item $\phi$ is torsion;
        \item $X^\phi$ is projective;
        \item there is a class $\alpha\in NS_{\mathbb Q}(X^\phi)$ such that $q(\alpha,\eta)\ne 0$.
    \end{enumerate}
\end{theorem}

\begin{proof}
    {\bf (1) $\Rightarrow$ (2).} By Theorem \ref{theorem iso between neron severis}, there is a relatively ample class on $X^\phi$. Hence $X^\phi$ is projective.

    \hfill

    {\bf (2) $\Rightarrow$ (3).} An ample class $\alpha$ on $X^\phi$ will do the job.

    \hfill

    {\bf (3) $\Rightarrow$ (1).} We can find a torsion element $\psi\in\Sha$ such that $\phi - \psi$ is arbitrarily close to $0$. In particular, we may assume that $X^{\phi-\psi}$ is K\"ahler. A cohomology class $\alpha$ has non-zero intersection with $\eta$ if and only if the restriction of $\alpha$ to a smooth fiber is non-trivial (Theorem \ref{theorem_restriction}). By Theorem \ref{theorem iso between neron severis} the manifold $X^{\phi-\psi}$ carries a rational $(1,1)$-class $\alpha'$ such that $q(\alpha',\eta) \ne 0$ as well. By scaling $\alpha'$ we may assume that $q(\alpha',\eta)=1$.

    We will consider $X^{\phi-\psi}$ as a degenerate twistor deformation $X_t$ of $X$, where $t\bar\sigma\in H^{0,2}(X)\simeq\C$ maps to $\phi-\psi$ under the projection $\C\to\Sha$. This enables us to identify $H^2(X)$ and $H^2(X^{\phi-\psi})$. Since $X$ is projective, there exists a rational class $\alpha\in H^{1,1}_{\mathbb Q}(X)$ such that $q(\alpha,\eta) = 1$ (Lemma \ref{lemma eta perp}). Hence, the class $l:=\alpha - \alpha'$ is a rational class orthogonal to $\eta$. This class satisfies
    \begin{gather*}
    q(l,\sigma) = q(l,\sigma + t\eta) = q(\alpha,\sigma + t\eta) - q(\alpha', \sigma + t\eta) = q(\alpha, \sigma + t\eta) =\\
    = q(\alpha, \sigma) + tq(\alpha, \eta) = tq(\alpha,\eta) = t.
    \end{gather*}
    By Lemma \ref{torsion in sha}, the class $[t\bar\sigma] = \phi - \psi\in\Sha$ is torsion, hence the same holds for $\phi$.  
\end{proof}

\begin{corollary}
\label{corollary all kahler except nowhere dense}
    Let $\pi\colon X\to B$ be a Lagrangian fibration on an irreducible hyperk\"ahler manifold. As before, denote by $\Sha'$ the subset of $\phi\in \Sha$ s.t. $N\phi\in\Sha^0$ for some $N\in\mathbb Z_{>0}$. Then the set of $\phi\in\Sha'$ such that $X^\phi$ is K\"ahler is open and dense in $\Sha'$.
\end{corollary}
\begin{proof}
    First, it is enough to prove this corollary for a projective $X$. Indeed, $X^\psi$ is projective for some $\psi\in\Sha^0$ by the same argument as the one used in the proof of Corollary \ref{corollary matsushita}. If we manage to prove Corollary \ref{corollary all kahler except nowhere dense} for $X^\psi$, then the same result for $X$ will follow because every Shafarevich-Tate twist of $X^\psi$ is a Shafarevich-Tate twist of $X$.

    Let us assume that $X$ is projective. Twists $X^\phi$ of $X$ with respect to torsion elements $\phi\in\Sha'$ are projective by Theorem \ref{theorem_kahler}(2). Moreover, the set of torsion elements is dense in $\Sha'$ (Corollary \ref{corollary dense}). Hence the set of K\"ahler twists with respect to $\phi\in \Sha'$ is dense in $\Sha'$. K\"ahlerness is open in any space of deformations, therefore, this set is also open.
\end{proof}



\section{Limits of Lagrangian fibrations}
\label{section limits of lagrangian fibrations}


\subsubsection{Douady spaces.} Theorem \ref{theorem rigidity} relies on careful examination of Douady spaces of hyperk\"ahler manifolds, which we will introduce below. The main reference for this subsection are lecture notes \cite{magnusson2005lectures}.

Douady showed that the set of complex subspaces of a complex space $X$ admits a structure of a complex space \cite[Theorem 6.0.2]{magnusson2005lectures}.

\begin{definition}
    The complex space parametrizing purely $k$-dimensional complex subspaces of a complex space $X$ is called the {\em Douady space} of $X$ and will be denoted by $\mathfrak D_k(X)$. 
\end{definition}
Similarly to the Hilbert scheme of a projective variety, the Douady space of $X$ represents the functor from complex spaces to sets sending $S$ to the set of flat families of $k$-dimensional subspaces of $X$ parametrized by $S$. A flat family of subspaces of $X$ over $S$ is a subspace $Y\subset S\times X$ having pure relative dimension $k$ over $S$, flat over $S$, and proper over $X$. In particular, the Douady space admits a {\em universal family} $Y\subset \mathfrak D_k(X)\times X$ of subspaces of $X$.

We collect here a few useful facts about Douady spaces. We will formulate them for K\"ahler spaces, which are potentially singular generalizations of K\"ahler manifolds \cite[Definition 4.2.3]{magnusson2005lectures}.


\begin{theorem}[{\cite[Corollary 11.3.2]{magnusson2005lectures}}]
\label{theorem liebermann}
    Let $X$ be a compact K\"ahler space. Then every connected component of $\mathfrak D_k(X)$ is compact.
\end{theorem}

\begin{theorem}
\label{theorem varouchas}
    Let $X$ be a compact K\"ahler space. Then every connected component of $\mathfrak D_k(X)$ is of Fujiki class $\mathcal C$, i.e. bimeromorphic to a K\"ahler manifold.
\end{theorem}

\begin{remark}
\label{remark barlet}
    The Douady space $\mathfrak D_k(X)$ admits a proper birational map to the Barlet space $\mathfrak B_k(X)$ of $X$ parametrizing $k$-dimensional cycles on $X$ \cite[Theorem 11.3.1]{magnusson2005lectures}. As was shown by Varouchas, every connected component of the Barlet space of a compact K\"ahler space is a K\"ahler space \cite[Theorem 11.1.6]{magnusson2005lectures}. It is an open question whether the Douady space of a compact K\"ahler space is K\"ahler. 
\end{remark}

The construction of a Douady space works in families, and Theorem \ref{theorem liebermann} holds in a relative setting.

\begin{theorem}[{\cite[Proposition 2.6]{greb2013lagrangian}}]
\label{theorem liebermann relative}
    Let $\mathcal X\to \mathbb D^n$ be a smooth family of compact complex manifolds over a polydisk. Suppose that the central fiber $X_0$ is K\"ahler. Then every connected component of the relative Douady space $\mathfrak D_k(\mathcal X/\mathbb D^n)\to \mathbb D^n$ is proper over $\mathbb D^n$ after possibly shrinking $\mathbb D^n$.
\end{theorem}

\begin{remark}
    The theorem above is formulated in \cite{greb2013lagrangian} for Barlet spaces, but by Remark \ref{remark barlet}, the same result holds for Douady spaces.
\end{remark}

In the following example the Douady space has a particularly easy description. The details of the proof are technical when the base is not assumed to be smooth.
\begin{lemma}
\label{lemma douady of morphism}
    Let $f\colon M\to S$ be a flat morphism of compact K\"ahler spaces. Assume that $S$ is reduced and irreducible and $h^0(F,\mathcal O_F)=1$ for any fiber $F$ of $f$. Let $T$ be the connected component of the Douady space of $M$ containing a fiber of $f$. Then $T$ is isomorphic to $S$ as a K\"ahler space.
\end{lemma}
\begin{proof}
    {\bf Step 1.} Let $U\subset T\times M$ be the universal family over the connected component $T$ of the Douady space. By the universal property of the Douady space, there is a morphism $i\colon S\to T$ such that the morphism $f\colon M\to S$ is the pullback of the projection $g\colon U\to T$ via the morphism $i$. The morphism $i$ is injective on points. We claim that it induces an isomorphism on tangent spaces. Indeed, the tangent space at a point of $T$ corresponding to a fiber $F=f^{-1}(s)$ is isomorphic to $Hom_F(\mathcal I_F/\mathcal I_F^2,\mathcal O_F)$, where $\mathcal I_F$ is the ideal sheaf of $F$ \cite[\S 9.1, R\'emarque 3]{douady1966probleme}. By flatness of $f$ this is isomorphic to 
    \[
    Hom_F(\mathfrak m_s/\mathfrak m_s^2\otimes\mathcal O_F, \mathcal O_F)\simeq T_sS\otimes H^0(F,\mathcal O_F)\simeq T_sS,
    \]
    where at the end we used the assumption that $h^0(F,\mathcal O_F)=1$. If $S$ is smooth, this is enough to conclude that $i$ is an isomorphism. Indeed, by the holomorphic inverse function theorem and injectivity of $i$, the morphism $i$ must be an open embedding. Since $S$ is proper, its image in $T$ is closed. Hence $i$ is an isomorphism.  If $S$ is not necessarily smooth,  we need to do additional work.

    \hfill

    {\bf Step 2.}  By semi-continuity, there exists an open subset $T'\subset T$ containing the image of $S$ such that for every $t\in T'$, $h^0(g^{-1}(t), \mathcal O_{g^{-1}(t)}) = 1$. In particular, all fibers of $g$ over $T'$ are connected. Denote the preimage of $T'$ in $U$ by $U'$ and the restriction of $g$ to $U'$ by $g'$. It follows from \cite[Ch.~5, Corollary 2]{mumford1974abelian} that the natural map $\mathcal O_{T'}\to g'_*\mathcal O_{U'}$ is an isomorphism.

    \hfill
    
    {\bf Step 3.} We claim that every subvariety $F'$ of $X$ corresponding to a point in $T'$ is set-theoretically contained in a fiber of $f$. Pick a K\"ahler form $\omega_Y$ on $Y$ and a K\"ahler form $\omega_X$ on $X$. Denote by $r$ the relative dimension of $X$ over $Y$. Then
    \[
    \int\limits_{F'}f^*\omega_Y\cdot \omega_X^{r-1} = \int\limits_{F}f^*\omega_Y\cdot \omega_X^{r-1} = 0,
    \]
    because the cohomology classes $[F']$ and $[F]$ coincide. The integral of $f^*\omega_Y\cdot \omega_X^{r-1}$ over any $r$-dimenisonal subvariety is non-negative, hence the integral of $f^*\omega_Y\cdot \omega_X^{r-1}$ over each irreducible component of $F'$ must vanish.  If an irreducible component $F''$ of $F'$ is not contracted to a point by $f$, then the form $f^*\omega_Y\cdot \omega_X^{r-1}$ is positive definite at any smooth point of $F''$. Therefore $ \int\limits_{F''}f^*\omega_Y\cdot \omega_X^{r-1}>0$, contradiction. We conclude that every irreducible component of $F'$ is contained in a fiber of $F$. Since $F'$ is connected, it is contained in a single fiber.

    \hfill

    {\bf Step 4.} We claim that there exists a retraction $r\colon T'\to S$ making the following diagram commutative. 
    \[
    \begin{tikzcd}
M \arrow[d, "f"] \arrow[r] & U' \arrow[d, "g'"] \arrow[r] & M \arrow[d, "f"] \\
S \arrow[r, "i"] & T' \arrow[r, "r"] & S
\end{tikzcd}
    \]
    Indeed, by Step 2, the map $g'\colon U'\to T'$ satisfies $g'_*\mathcal O_{U'}\simeq \mathcal O_{T'}$ and by Step 3, the composition map $U'\to M\to S$ contracts the fibers of $g'$. Together these two properties imply the existence of a morphism $r\colon T'\to S$ as desired \cite[Ch.~10, \S 6.1]{grauert2012coherent}. This morphism is identity on the image of $S$ in $T'$.

    \hfill

    {\bf Step 5.} Finally, we will show that the maps $r\colon T'\to S$ and $i\colon S\to T'$ are inverse to each other. Indeed, since $r$ is a rectraction, we have a direct sum decomposition 
    \begin{equation}
    \label{equation direct sum}
    i^{-1}\mathcal O_{T'}\simeq \mathcal O_S\oplus i^{-1}\mathcal J
    \end{equation}
    of sheaves on $S$. Here $\mathcal J\subset \mathcal O_{T'}$ denotes the ideal sheaf of $S$ inside $T'$. It follows from the direct sum decomposition (\ref{equation direct sum}) that $\forall s\in S$ the maximal ideal $\mathfrak m_{i(s)}\subset \mathcal O_{i(s), T'}$ satisfies
     \[
     \mathfrak m_{i(s)} \simeq \mathfrak m_s\oplus \mathcal I_{i(s)}.
     \]
    As was shown in Step 1, the map $i\colon S\to T'$ induces an isomorphism on tangent spaces. Hence the kernel of the dual map on cotangent spaces
    \[
    \mathfrak m_{i(s)}/\mathfrak m_{i(s)}^2\to \mathfrak m_{s}/\mathfrak m_{s}^2
    \]
    is isomorphic to $\mathcal J_{i(s)}/\mathfrak m_{i(s)} \mathcal J_{i(s)}$. By assumption the kernel is zero. Nakayama's lemma implies that $\mathcal J_{i(s)}=0$ $\forall s\in S$, hence $\mathcal I=0$. It follows that $i\colon S\to T'$ is an isomorphism. Therefore $i\colon S\to T$ is an open embedding. The image of $S$ is closed in $T$ because $S$ is proper, hence $i\colon S\to T$ is an isomorphism.
\end{proof}

\subsubsection{Douady space of fibrations.}\label{subsub douady space} Consider a smooth family of K\"ahler manifolds $\mathcal X\to \mathbb D$ and denote by $X$ its central fiber. Suppose that over an open subset $U\subset \mathbb D$ with $0\in \overline U$ the family $\mathcal X_U$ admits a family of flat morphisms $\Pi_U\colon\mathcal X_U\to \mathcal B_U$ with connected fibers. Assume that every fiber $F$ of $\Pi_U$ satisfies $h^0(F,\mathcal O_F)=1$, which is the case for families of Lagrangian fibrations by Remark \ref{remark h0 is one}.

Let $\mathcal B$ be the irreducible component of the relative Douady space of $\mathcal X$ containing a fiber of $\Pi_U$. By Theorem \ref{theorem liebermann relative}, $\mathcal B$ is proper over $\mathbb D$. In particular, the central fiber of $\mathcal B$ is non-empty; hence there exists a subvariety $F\subset X$ deformation equivalent to a fiber of $\Pi_U$. Let $\mathcal Y\subset\mathcal B\times_{\mathbb D}\mathcal X$ be the universal family of subvarieties over $\mathcal B$. Denote its projection to $\mathcal X$ by $G$. The map $G$ is an isomorphism over $\mathcal X_U$ by Lemma \ref{lemma douady of morphism}. It follows that $G\colon \mathcal Y\to \mathcal X$ is a proper bimeromorphic morphism. In particular, $\mathcal Y\to\mathcal X$ is surjective, hence deformations of $F$ cover $X$. 

Denote by $Y$ the strict preimage of $X\subset \mathcal X$ in $\mathcal Y$ and by $B$ the image of $Y$ under the projection  $\Pi\colon \mathcal Y\to \mathcal B$. Let us denote by $g\colon Y\to X$ the restriction of $G$ to $Y$ and by $\pi_Y\colon Y\to B$ the projection of $Y$ to $B$. Note that the central fiber $Y_0$ of $\mathcal Y$ may contain irreducible components other than $Y$. 

\begin{lemma}
\label{lemma smoothness of douady}
    The family $\mathcal Y$ is smooth at a general point of $Y\subset Y_0$ and the family $\mathcal B\to \mathbb D$ is smooth at a general point of $B\subset B_0$.
\end{lemma}

\begin{proof}
    The variety $\mathcal Y$ is isomorphic to $\mathcal X$ outside the exceptional divisor of $G$, hence smooth at a general point of $Y$. Moreover, the morphism $\mathcal Y\to \mathbb D$ is smooth at a general point of $Y$, i.e., $Y$ is a reduced irreducible component of $Y_0$. It follows that the irreducible component $B$ of the central fiber of $\mathcal B$ must be reduced too. Indeed, denote by $t$ the pullback of the coordinate function on $\mathbb D$ to $\mathcal B$. If $t$ decomposes as $t = s^k$, $k>1$, at a general point of $B$, then $\Pi^*t = (\Pi^*s)^k$ at a general point of $Y$, which contradicts that $Y$ is reduced.
\end{proof}



\begin{lemma}
\label{lemma connected}
The morphism $\pi_Y\colon Y\to B$ has connected fibers over a non-empty Zariski dense subset of $B$.
\end{lemma}
\begin{proof}
    The morphism $\Pi\colon \mathcal Y\to \mathcal B$ has connected fibers over $\mathcal B_U$. By Zariski main theorem it has connected fibers over the locus where $\mathcal B$ is normal. The family $\mathcal B\to \mathbb D$ is smooth at a general point of $B$ by Lemma \ref{lemma smoothness of douady}, hence normal.
\end{proof}

\subsubsection{Isotropic divisors.}
\label{subsub isotropic divisor} From now on, we'll focus on the case where $\mathcal X\to \mathbb D$ is a family of irreducible hyperk\"ahler manifolds whose restriction to an open subset $U$ with $0\in \overline U$ admits a family of Lagrangian fibrations $\Pi_U\colon \mathcal X_U\to \mathcal B_U$. We'll keep the notation of \ref{subsub douady space}.


\begin{lemma}
\label{lemma fibers lagrangian}
    Let $\mathcal X\to \mathbb D$ be a family of irreducible hyperk\"ahler manifolds as above. Then a general fiber $F$ of the fibration $\pi_Y\colon Y\to B$ is Lagrangian with respect to the form $g^*\sigma$ on $Y$, i.e., the restriction of $g^*\sigma$ to the smooth locus of $F$ is trivial.
\end{lemma}
\begin{proof}
    By Lemma \ref{lemma smoothness of douady}, a general fiber of $\pi_Y\colon Y\to B$ has a smooth point. The sheaf $\Omega^2_{\mathcal Y/\mathcal B}$ is locally free at the locus of points of $\mathcal Y^{reg}$ where $\Pi$ is smooth, hence it is generically free over a general fiber of $\pi_Y$. Denote by $\tilde\sigma$ the holomorphic $2$-form on $\mathcal X$ whose restriction to $X_t$ is the holomorphic symplectic form $\sigma_t$. The $2$-form $G^*\tilde\sigma$ induces a section of $\Omega^2_{\mathcal Y/\mathcal B}$ which is trivial over $\mathcal X_U$. It follows that the restriction of $G^*\tilde\sigma$ to the smooth locus of a general fiber $F$ of $\pi_Y$ is trivial.
\end{proof}

Let $\eta\in H^2(X_t,\Z)$ be the pullback of a very ample class from the base of the Lagrangian fibration $\pi_t\colon X_t\to B_t$ for $t\in U$. Since the family $\mathcal X\to \mathbb D$ is smooth, the cohomology groups $H^2(X_t,\Z)$ and $H^2(X,\Z)$ are naturally identified. By a slight abuse of notation we will denote by $\eta$ the corresponding class in $H^2(X,\Z)$. We may apply the same argument as in \ref{subsub douady space} to the irreducible component of the relative Douady space of $\mathcal X$ containing divisors with class $\eta$ to conclude that every point $x\in X$ lies in a divisor $D$ with class $\eta$.

\begin{lemma}
\label{lemma vertical}
Let $\mathcal X\to \mathbb D$ be a family of irreducible hyperk\"ahler manifolds as in Theorem \ref{theorem rigidity}. Keep the notation of \ref{subsub douady space}. Then for any divisor $D$ on $X$ with class $\eta\in H^2(X,\Z)$, the divisors $g^{-1}(D)$ and $\pi_Y^{-1}(\pi_Y(g^{-1}(D)))$ coincide set-theoretically over a non-empty Zariski open subset of $B$.

\end{lemma}
\begin{proof}
    By assumption, the central fiber $X$ of $\mathcal X\to\mathbb D$ is not projective. By Lemma \ref{lemma eta perp}, all curves on $X$ are orthogonal to $\eta$. Hence for every curve $C$ on $Y$
    \begin{equation}
    \label{equation zero}
    g^*\eta\cdot C = \eta\cdot g_*C = 0.        
    \end{equation}

    A general fiber $F$ of the fibration $\pi_Y\colon Y\to B$ maps isomorphically onto a Lagrangian subvariety of $X$ (Lemma \ref{lemma fibers lagrangian}). By \cite[Theorem 1.1]{lehn2016deformations}, a Lagrangian subvariety of a hyperk\"ahler manifold $X$ is projective. Pick an effective divisor $D$ on $X$ with class $\eta$. For any irreducible component $F'$ of $F$, the divisor $g^{-1}(D)$ either contains $F'$, or the restriction $g^*D\restrict{F}$ is an effective numerically trivial divisor on $F'$. An effective numerically trivial divisor on a projective variety is trivial. Therefore, any irreducible component $F'$ of $F$ is either contained in $g^{-1}(D)$ or is disjoint from $g^{-1}(D)$. 
    
    Since a general fiber $F$ is connected (Lemma \ref{lemma connected}), if at least one component is contained in $g^{-1}(D)$, then all of them are. We obtain that a general fiber $F$ of $\pi_Y$ is either contained in $g^{-1}(D)$ or disjoint from $g^{-1}(D)$. The claim follows.
\end{proof}

\begin{remark}
    Note that Lemma \ref{lemma vertical} may not hold when $X$ is projective. In this case the divisor $g^*D$ may contain components supported on the exceptional divisor of $g\colon Y\to X$. 
\end{remark}

\begin{lemma}
\label{lemma a lot of divisors}
Let $\mathcal X\to \mathbb D$ be a family of irreducible hyperk\"ahler manifolds as in Theorem \ref{theorem rigidity}. Then for a general point $b\in B^{reg}$ and a general hyperplane $H\subset T_bB$, there exists a divisor $D$ on $X$ with class $\eta$ such that $g^{-1}(D)$ contains $\pi_Y^{-1}(b)$ and $\pi_Y(g^{-1}(D))$ is tangent to $H$.
\end{lemma}
\begin{proof}

The projection $\Pi\colon \mathcal Y\to \mathcal B$ induces a map $\mathcal X\setminus \mathrm{Exc}(G)\to \mathcal B$. Denote by $\mathcal B^\circ$ the open subset of $\mathcal B^{reg}$ where the family $\mathcal B\to \mathbb D$ is smooth and let $\mathcal X^\circ$ be its preimage in $\mathcal X\setminus\mathrm{Exc}(G)$. By Lemma \ref{lemma smoothness of douady}, the subset $\mathcal B^\circ$ contains a general point of $B\subset \mathcal B$. 

Let $D$ be a divisor on $X$ with class $\eta\in H^2(X,\Z)$ and $\mathcal P$ the irreducible component of the Douady space of $\mathcal X$ containing $D$ that dominates $\mathbb D$. We will denote the universal family of divisors over $\mathcal P$ by $\mathcal D\subset \mathcal P\times_{\mathbb D}\mathcal X$. Consider the open subset $\mathcal D^\circ$ of $\mathcal D$ consisting of points $(D,x)\in \mathcal P\times_{\mathbb D}\mathcal X^\circ$ such that $D$ is smooth at $x$. We have a well-defined map
\begin{equation}
\label{equation hyperplanes}
\mathcal D^\circ\to \mathbb P(T_{\mathcal X^\circ/\mathbb D})
\end{equation}
sending a point $(D,x)$ to the tangent space of $D$ in $x$.\footnote{For us the projectivization $\mathbb P(\mathcal E)$ of a locally free sheaf $\mathcal E$ is the set of hyperplanes in $\mathcal E$. }

By Lemma \ref{lemma vertical}, for any divisor $D$ on $\mathcal X$ with class $\eta$, the intersection of $D$ with $\mathcal X^\circ$ is set-theoretically the union of fibers of $\Pi$. It follows that the map (\ref{equation hyperplanes}) factors through a map
\begin{equation}
\label{equation hyperplanes base}
    \mathcal D^\circ\to \mathbb P(T_{\mathcal B^\circ/\mathbb D}).
\end{equation}

Up to possibly replacing $\eta\in H^2(X_t,\Z)$ with its multiple, we may assume that $\eta$ is the pullback of a very ample class from $B_t$. Therefore, the map (\ref{equation hyperplanes base}) is surjective over the open subset $U\subset \mathbb D$ parametrizing manifolds with a Lagrangian fibration. 

Let $D$ be a divisor on $X$ with class $\eta$. We claim that for a sufficienly large $n\in \Z$, a generic divisor in the linear system $|nD|$ has a reduced component. Indeed, by Bertini's theorem, the moving part of a generic divisor in the linear system $|nD|$ is the pullback of a smooth divisor under the rational map
\[
\phi_{|nD|}\colon X\dashrightarrow \mathbb P(H^0(\mathcal O_X(nD)).
\]
Therefore, the moving part is reduced. We conclude that the open subset $\mathcal D^\circ\subset\mathcal D$ maps surjectively onto $\mathbb D$, after possibly replacing $\eta$ with its multiple. Hence a general point of the projectivization $\mathbb P(T_{B\cap \mathcal B^\circ})$ of the tangent bundle of $B\cap\mathcal B^\circ$ is contained in the image of the map (\ref{equation hyperplanes base}). The claim follows.
\end{proof}

\subsubsection{Proof of Theorem \ref{theorem rigidity}.}
\label{subsub proof of theorem rigidity}
\begin{proof}
{\bf Step 1.} We will show that the exceptional divisor of the map $g\colon Y\to X$ doesn't have horizontal irreducible components, i.e., components that map surjectively onto $B$. Let $E$ be a horizontal component of the exceptional divisor. Pick a general irreducible curve $C$ on $E$ contracted by $g$. Since a fiber of $\pi_Y$ maps isomorphically onto its image in $X$, the curve $C$ is not contained in a fiber. Pick a general point $x\in C$ and a hyperplane $H\in T_{\pi_Y(x)}B$ transversal to $\pi_Y(C)$. By Lemma \ref{lemma a lot of divisors}, there exists a divisor $D$ on $X$ with class $\eta$ such that $g^{-1}(D)$ contains the fiber $\pi_Y^{-1}(\pi_Y(x))$ and is tangent to $H$. Therefore $g^*D\cdot C>0$. This contradicts equation (\ref{equation zero}).

\hfill

{\bf Step 2.} Since the exceptional divisor of $g$ has no horizontal components, the fibration $\pi_Y$ induces an almost holomorphic fibration $\pi\colon X\dashrightarrow B$ on $X$. By definition, a rational map $\pi\colon X\dashrightarrow B$ is called an almost holomorphic fibration if there exists a Zariski open subset $X'\subset X$ such that the map $\pi$ is defined and proper on $X'$ \cite[Definition 1]{matsushita2014almost}. By Arnold-Liouville theorem, a smooth fiber of an almost holomorphic fibration on a holomorphic symplectic manifold is a Lagrangian torus. By \cite[Theorem 4.1]{greb2013lagrangian}, a Lagrangian torus on a non-projective irreducible hyperk\"ahler manifold is a fiber of a Lagrangian fibration. Hence the rational map $\pi$ extends to a Lagrangian fibration
$$
\pi\colon X\to B'.
$$
The connected component of the Douady space of a fiber of $\pi$ is $B'$ (Lemma \ref{lemma douady of morphism}) and the universal family is isomorphic to $X$. It follows that the birational map $G\colon \mathcal Y\to \mathcal X$ is an isomorphism, and hence $\mathcal X$ is a family of Lagrangian fibrations.
\end{proof}

\section{K\"ahler twists}
\label{sec kahler twists}

As we showed in Corollary \ref{corollary all kahler except nowhere dense}, all twists $X^\phi$ with respect to $\phi\in\Sha'$ are K\"ahler except maybe for a nowhere dense subset of $\Sha'$. In this section we will show that $X^\phi$ is actually K\"ahler for all $\phi\in\Sha'$, and thus we prove Theorem \ref{theorem_kahler}(1). Note that Theorem \ref{theorem_kahler}(1) will immediately follow from the statement below by applying it to $X^\phi$ for some $\phi\in\Sha'$. 

\begin{theorem}
\label{theorem deformations are kahler replacement}
    Let $\pi\colon X\to B$ be a Lagrangian fibration on an irreducible holomorphic symplectic manifold. Consider the restriction $\mathcal X\to \mathbb D$ of its Shafarevich--Tate family to a disk $\mathbb D\subset \mathbb A^1$. Suppose that the set $U\subset \mathbb D$ parametrizing K\"ahler Shafarevich--Tate deformations of $X$ is non-empty and $0\in\overline{U}$. Then $X$ is hyperk\"ahler.
\end{theorem}

\subsubsection{Limits of hyperk\"ahler manifolds}

It follows from Corollary \ref{corollary all kahler except nowhere dense} that every Shafarevich--Tate twist $X^\phi$ with respect to $\phi\in\Sha'$ is a {\em limit of hyperk\"ahler manifolds} in the sense of the following definition.

\begin{definition}
    Let $X$ be a compact complex manifold. Consider a family of deformations $\mathcal X\to T$ of $X$, and let $0\in T$ be the point corresponding to $X$. The manifold $X$ is said to be a {\em limit of K\"ahler manifolds} if $X_t$ is K\"ahler for a general $t\in U$.
\end{definition}

A limit of K\"ahler manifolds does not have to be K\"ahler, however the following is expected to be true.

\begin{conjecture}\cite{popovici2011deformation}
    A limit of K\"ahler manifolds is Fujiki class $\mathcal C$, i.e., is bimeromorphic to a K\"ahler manifold.
\end{conjecture}




\subsubsection{Idea of the proof}
\label{subsection idea of the proof}

Before we get started with the proof of Theorem \ref{theorem deformations are kahler replacement}, we will sketch its main steps below.

\subsubsection*{Step 1. Period map and Torelli theorems.}\label{subsub step 1} (Subsection \ref{subsection period map and torelli theorems}). Using Local and Global Torelli Theorems (Theorem \ref{theorem global torelli}), we construct a family
$$
\mathcal Y\to \mathbb D
$$
such that $Y_t$ is hyperk\"ahler for all $t\in \mathbb D$ and $X_t$ is bimeromorphic to $Y_t$ for all $t\in U\subset \mathbb D$ (Lemma \ref{lemma birational family}).

\subsubsection*{Step 2. Lagrangian fibration on $Y_t$.}\label{subsub step 2} (Subsection \ref{subsection lagrangian fibrations on non-projective hyperkahler manifolds}). Let $t$ be a very general point in $U$. We will show in Corollary \ref{corollary admits lagrangian fibration} that $Y_t$ admits a Lagrangian fibration $p_t\colon Y_t\to B'$, and every bimeromorphism $f_t\colon X_t\dashrightarrow Y_t$ commutes with the Lagrangian fibrations on $X_t$ and $Y_t$. This step relies on a result by Greb-Lehn-Rollenske \cite[Theorem 4.1]{greb2013lagrangian} that says that a non-projective hyperk\"ahler manifold containing a Lagrangian torus admits a Lagrangian fibration.

\subsubsection*{Step 3. $\mathcal Y\to \mathbb D$ is a Shafarevich--Tate family.}\label{subsub step 3} (Subsection \ref{subsection almost shafarevich tate family}). We will see in Proposition \ref{proposition almost shafarevich tate family} that the family $\mathcal Y\to \mathbb D$ is a Shafarevich--Tate family after restriction to some open dense subset $V\subset U$. We will use Theorem \ref{theorem rigidity} to conclude that $\mathcal Y\to \mathbb D$ is a Shafarevich--Tate family over the entire $\mathbb D$. Moreover, it will turn out that the base $B'$ of the Lagrangian fibration $p_t\colon Y_t\to B'$ for $t\in V$ is isomorphic to $B$ (Proposition \ref{proposition bases are isomorphic}).

\subsubsection*{Step 4. Chamber decomposition of the positive cone.}\label{subsub step 4} (Subsection \ref{subsection period map revisited}) We will formulate everything we have shown so far in terms of the period map. Next, we will use Markman's description of the fibers of the period map in terms of K\"ahler chambers of the positive cone \cite{markman2011survey} to conclude the result.

\subsection{Period map and Torelli theorems}
\label{subsection period map and torelli theorems}

\subsubsection{Period map for hyperk\"ahler manifolds.} 
Let $X$ be a hyperk\"ahler manifold and $\Lambda$ be a lattice isomorphic to the lattice $(H^2(X,\Z),q_X)$, where $q_X$ is the BBF form (Definition \ref{definition bbf}). Denote $\Lambda_\C:=\Lambda\otimes \C$.

\begin{definition}
\label{definition moduli space}
    The {\em moduli space $\mathcal M_\Lambda$ of $\Lambda$-marked hyperk\"ahler manifolds} is the moduli space of pairs $(Y,g)$ where $Y$ is a hyperk\"ahler manifold and $g\colon H^2(Y,\Z)\to \Lambda$ is an isomorphism of lattices.
\end{definition}

\begin{definition}
\label{definition period map}
    The {\em period map}
    $$
    \operatorname{Per}\colon \mathcal M_\Lambda\to \mathbb P(\Lambda_\C)
    $$
    sends the point of $\mathcal M_\Lambda$ corresponding to a pair $(Y,g)$ to the class of the line $g(H^{2,0}(Y))\subset \Lambda_\C$. The image of a pair $(Y,g)$ under the period map is called its {\em period}.
\end{definition}

\begin{theorem}
\label{theorem global torelli}
    \begin{enumerate}
        \item The image of the period map is contained in the subset $\Omega_\Lambda$ consisting of $[\sigma]\in \mathbb P(\Lambda_\C)$ such that 
        $$
            q(\sigma) = 0\:\:\:\:\: and \:\:\:\:\: q(\sigma,\overline\sigma) >0.
        $$
        \item \text{(Local Torelli Theorem \cite{beauville1983varietes})} The period map is a local biholomorphism onto $\Omega_\C$.
        \item \text{(Global Torelli Theorem \cite[Theorem 8.1]{huybrechts1999compact}, \cite[Corollary 6.1]{huybrechts2011global})}. Let $\mathcal M^0_\Lambda$ be a connected component of $\mathcal M_\Lambda$. Then the period map 
        $$
        \operatorname{Per}\colon \mathcal M^0_\Lambda\to \Omega_\Lambda
        $$
        is surjective. Moreover, two points $(X,g)$ and $(X',g')$ of $\mathcal M^0_\Lambda$ have the same periods if and only if there exists a bimeromorphism $f\colon X\dashrightarrow X'$ such that the pullback map $f^*\colon H^2(X')\to H^2(X)$ coincides with $g^{-1}\circ g$.
    \end{enumerate}
\end{theorem}

\subsubsection{Period map for Shafarevich--Tate deformations.}\label{subsub period map} Assume that $X$ admits a Lagrangian fibration $\pi\colon X\to B$. Consider its Shafarevich--Tate family (Definition \ref{definition sht family})
$$
\Pi\colon \mathcal X\to \mathbb A^1.
$$
We can construct a period map
$$
\operatorname{Per}_{\Sha T}\colon \mathbb A^1 \to \mathbb P(H^2(X,\C)).
$$
exactly as in Definition \ref{definition period map} by sending the class of $t\in\mathbb A^1$ to the class of the holomorphic symplectic form $\sigma_t$ on $X_t$. Denote by $\eta$ the class of the pullback of an ample class on $B$ to $X$. It is easy to see \cite[Proposition 3.9]{abasheva2025shafarevich} that the map $Per_{\Sha T}$ is an isomorphism onto the affine line
\begin{equation}
L := \{[\sigma + t\eta]\mid t\in\C\}\subset \mathbb P(H^2(X,\C)).
\end{equation}
In particular, the image of $\operatorname{Per}_{\Sha T}$ lies in $\Omega_\Lambda$.
\begin{definition}
    We will refer to $L\subset \Omega_\Lambda$ as the {\em Shafarevich--Tate line} of the Lagrangian fibration $\pi\colon X\to B$. 
\end{definition}
\begin{lemma}
\label{lemma birational family}
    As in Theorem \ref{theorem deformations are kahler replacement}, let $\mathcal X\to \mathbb D$ be a Shafarevich--Tate deformation over a disk $\mathbb D\subset \mathbb A^1$. Assume that the closure of the subset $U\subset\mathbb D$ consisting of K\"ahler Shafarevich-Tate twists contains $0$. Then there exists a family $\mathcal Y\to \mathbb D$ such that
    \begin{itemize}
        \item $\forall t\in\mathbb D$, $Y_t$ is hyperk\"ahler;
        \item $\forall t\in U$, the manifolds $X_t$ and $Y_t$ are bimeromorphic
    \end{itemize}
\end{lemma}
\begin{proof}
    Let $X_t$ by a hyperk\"ahler Shafarevich--Tate deformation of $X_0$. By Global Torelli theorem (Theorem \ref{theorem global torelli}(3)) there exists a hyperk\"ahler manifold $Y_0$ deformation equivalent to $X_t$ whose period coincides with the period of $X_0$. The period map is a biholomorphism in a neighborhood of $Y_0$ in $\mathcal M_\Lambda$ (Theorem \ref{theorem global torelli}(2)). Hence we can find a family
    $$
    \mathcal Y\to \mathbb D
    $$
    of hyperk\"ahler manifolds such that its image under the period map coincides with the image of $\mathcal X\to \mathbb D$. For every $t\in U\subset \mathbb D$, the manifolds $X_t$ and $Y_t$ are deformation equivalent hyperk\"ahler manifolds whose periods coincide. Hence they are bimeromorphic (Theorem \ref{theorem global torelli}(3)).
\end{proof}

We are done with Step 1 (\ref{subsub step 1}) of the proof of Theorem \ref{theorem deformations are kahler replacement}.


\subsection{Lagrangian fibrations on non-projective hyperk\"ahler manifolds}
\label{subsection lagrangian fibrations on non-projective hyperkahler manifolds}



Recall that a hyperk\"ahler manifold $X$ with a Lagrangian fibration $\pi\colon X\to B$ is non-projective if and only if $NS(X)\subset \eta^\perp$ Lemma \ref{lemma eta perp}.
\begin{lemma}
\label{lemma all curves in fibers}
    Let $\pi\colon X\to B$ be a Lagrangian fibration on a non-projective hyperk\"ahler manifold. Then all curves on $X$ lie in fibers of $\pi$.
\end{lemma}
\begin{proof}
    Let $C\subset X$ be a curve. Denote by $c\in H^2(X,\Q)$ the class BBF dual to $C$.
    Then
    $$
    \eta\cdot C = q(\eta,c) = 0.
    $$
    Since $\eta = \pi^*h$ for an ample class $h$, we have
    $$
    h\cdot \pi_*C = 0.
    $$
    The class $h$ is ample, hence $\pi_*C$ is a trivial cycle. Therefore, $C$ is contained in a fiber of $\pi$.  
\end{proof}

\begin{proposition}
\label{proposition birational lagrangian fibrations}
    Let $f\colon X\dashrightarrow Y$ be a bimeromophism of irreducible hyperk\"ahler manifolds. Suppose that $X$ admits a Lagrangian fibration $\pi\colon X\to B$ and $X$ is not projective. Then the following holds.
    \begin{enumerate}
        \item The hyperk\"ahler manifold $Y$ admits a Lagrangian fibration $p\colon Y\to B'$.
        \item There exists a birational map $g\colon B\dashrightarrow B'$ making the diagram 
        $$
\begin{tikzcd}
X \arrow[r, "f", dashed] \arrow[d, "\pi"] & Y \arrow[d, "p"] \\
B \arrow[r, "g", dashed]                  & B'                 
\end{tikzcd}
        $$
        commutative.
        \item The meromoprhic map $f$ is holomorphic on $X^\circ$ and induces an isomorphism $X^\circ\to Y^\circ$. As before, $X^\circ$ (resp. $Y^\circ$) denotes the union of smooth fibers of $\pi$ (resp. $p$).
    \end{enumerate}
\end{proposition}

\begin{proof}
    {\bf Step 1.} First, we will show that $f$ is defined on $X^\circ$ and sends it isomorphically onto its image. Since $X$ and $Y$ are bimeromorphic, we can find a complex manifold $Z$ together with bimeromorphic maps $p\colon Z\to X$ and $q\colon Z\to Y$ making the following diagram commutative
    $$
\begin{tikzcd}
                          & Z \arrow[ld, "p"'] \arrow[rd, "q"] &   \\
X \arrow[rr, "f", dashed] &                                    & Y.
\end{tikzcd}
    $$
    For every $y\in Y$, the preimage $q^{-1}(y)\subset Y$ is rationally connected (see \cite{hacon2007shokurov} for the algebraic version of this theorem and \cite[Theorem 5]{fujino2023rational} for the analytic version). Let $C\subset q^{-1}(y)$ be a rational curve. Then either $C$ is contracted by $p$ or $p(C)$ is contained in the preimage of the discriminant locus $\pi^{-1}(\Delta)$. Indeed, by Lemma \ref{lemma all curves in fibers} there are no rational curves in $X$ passing through a point in $X^\circ = X\setminus \pi^{-1}(\Delta)$. If for some rational curve $C\subset q^{-1}(y)$, the image of $C$ in $X$ lies in $\pi^{-1}(\Delta)$, then the image of $q^{-1}(y)$ under $p$ lies in $\pi^{-1}(\Delta)$. Similarly, if some rational curve $C\subset q^{-1}(y)$ is contracted by $p$ to a point in $X^\circ$ then all rational curves in $q^{-1}(y)$ and hence the entire $q^{-1}(y)$ are contracted to this point. 


    Denote by $N\subset Y$ the image of $p^{-1}(\pi^{-1}(\Delta)))$ in $Y$. We have just shown that $p(q^{-1}(N)) = \pi^{-1}(\Delta)$ and $p(q^{-1}(Y\setminus N)) = X^\circ$. Moreover, all fibers of $q\restrict{q^{-1}(Y\setminus N)}$ are contracted by $p$. Therefore, the inverse rational map $f^{-1}\colon Y\dashrightarrow X$ is defined on $Y\setminus N$ and maps it to $X^\circ$. 

    We can choose holomorphic symplectic forms $\sigma_X$ and $\sigma_Y$ on $X$ and $Y$ respectively in such a way that 
    $(f^{-1})^*\sigma_X = \sigma_Y$. Since both forms $\sigma_X$ and $\sigma_Y$ are non-degenerate, the morphism $f^{-1}\restrict{Y\setminus N}\colon Y\setminus N\to X^\circ$ has $0$-dimensional fibers, hence is an isomorphism. That clearly implies that the map $f\restrict{X^\circ}$ is an isomorphism from $X^\circ$ onto $Y\setminus N$.

    \hfill

    {\bf Step 2.} The manifold $Y$ contains an open subset isomorphic to $X^\circ$, in particular, it contains a Lagrangian torus. Moreover, $Y$ is non-projective because it is bimeromorphic to a non-projective K\"ahler manifold. By Greb-Lehn-Rollenske theorem \cite[Theorem 4.1]{greb2013lagrangian}, the Lagrangian fibration $Y\setminus N\simeq X^\circ\to B^\circ$ extends to a Lagrangian fibration
    $$
    p\colon Y\to B'.
    $$
    Moreover, the base $B'$ is birational to $B$. The statement is proven.    
\end{proof}

\begin{corollary}
\label{corollary admits lagrangian fibration}
    In the notation of Lemma \ref{lemma birational family}, let $Y_t$ be the fiber of $\mathcal Y\to \mathbb D$ over a very general $t\in U$. Then $Y_t$ admits a Lagrangian fibration $p_t\colon Y_t\to B'$ bimeromorphic to the Lagrangian fibration $\pi_t\colon X_t\to B$.
\end{corollary}
\begin{proof}
    For a very general $t\in U$ the manifold $X_t$ is non-projective and satisfies $NS(X_t)\subset \eta^\perp$ (Lemma \ref{lemma eta perp}). By construction of $\mathcal Y$ (Lemma \ref{lemma birational family}), the manifolds $X_t$ and $Y_t$ are bimeromorphic for every $t\in U$. The statement of the corollary follows by applying Proposition \ref{proposition birational lagrangian fibrations} to $X_t$ and $Y_t$.
\end{proof}

We are done with Step 2 (\ref{subsub step 2}).


\subsection{$\mathcal Y\to \mathbb D$ is a Shafarevich--Tate family}
\label{subsection almost shafarevich tate family}

\begin{proposition}
\label{proposition almost shafarevich tate family}
    The family $\mathcal Y\to \mathbb D$ is a Shafarevich--Tate family of $Y=Y_0$.
\end{proposition}
\begin{proof}
   {\bf Step 1.} First, we'll show that the statement holds over an open subset $V\subset U$. For any point $t\in U$ such that $Y_t$ is non-projective, the manifold $Y_t$ admits a Lagrangian fibration (Corollary \ref{corollary admits lagrangian fibration}). We claim that the family $\mathcal Y$ coincides with the Shafarevich--Tate family of $Y_t$ in a neighborhood of $t$. The images under the period map of the Shafarevich--Tate family and of $\mathcal Y\to \mathbb D$ coincide. By the Local Torelli theorem (Theorem \ref{theorem global torelli}(2)), these families must coincide in a neighborhood of $t\in U$. The set of points $t\in U$ such that $Y_t$ is non-projective is dense, hence the family $\mathcal Y$ is a Shafarevich--Tate family over an open dense subset $V\subset U$.

   \hfill

   {\bf Step 2.} By Theorem \ref{theorem rigidity}, the Shafarevich--Tate family $\Pi_V\colon \mathcal Y_V\to B'\times V$ over $V$ extends to a family of Lagrangian fibrations $\Pi\colon\mathcal Y\to \mathcal B$. We have two family of deformations of $Y_0$: its Shafarevich--Tate family of $Y= Y_0$ and $\mathcal Y$. Their periods coincide, hence the two families coincide by the Local Torelli theorem (Theorem \ref{theorem global torelli}(2)).
\end{proof}

\begin{proposition}
\label{proposition bases are isomorphic}
    The bases $B$ and $B'$ of Lagrangian fibrations $\pi_t\colon X_t\to B$ and $p_t\colon Y_t\to B'$ are isomorphic for all $t\in V$.
\end{proposition}
\begin{proof}
    Let $t\in V\subset \mathbb D$ be such that $X_t$, and hence also $Y_t$, is projective. The manifolds $X_t$ and $Y_t$ are birational by construction of $\mathcal Y$ (Lemma \ref{lemma birational family}). This birational isomorphism preserves the class $\eta$, hence commutes with Lagrangian fibrations. 
    By \cite[Corollary 2]{matsushita2014almost}, the bases of birational Lagrangian fibrations on projective irreducible holomorphic symplectic manifolds are isomorphic. Hence $B\simeq B'$. 
\end{proof}

The two propositions of this subsection complete the proof of Step 3 (\ref{subsub step 3}) of Theorem \ref{theorem deformations are kahler replacement}.

\subsection{Period map revisited}
\label{subsection period map revisited}

\subsubsection{Preimage of the Shafarevich--Tate line under the period map.} Let $\pi\colon X\to B$ be a Lagrangian fibration on an irreducible hyperk\"ahler manifold. Introduce a marking $f\colon H^2(X,\Z)\xrightarrow{\sim}\Lambda$ on $X$ and let $\mathcal M^0_\Lambda$ be the connected component of the moduli space of $\Lambda$-marked irreducible hyperk\"ahler manifolds containing $(X,f)$ (Definition \ref{definition moduli space}). In \ref{subsub period map} we showed that the image of the Shafarevich--Tate family of $(X,f)$ under the period map
$$
\operatorname{Per}\colon \mathcal M^0_\Lambda\to \Omega_\Lambda\subset\mathbb P(\Lambda_\C)
$$
is an open subset of the Shafarevich--Tate line
\begin{equation}
\label{equation shafarevich-tate line}
L:=\{[\sigma+t\eta]\mid t\in \C\},
\end{equation}
where $\eta$ is the class of the Lagrangian fibration $\pi$.  For every $t\in \C$, we will denote by $\Lambda^{1,1}_t\subset \Lambda$ the orthogonal to the subspace generated by $\sigma+t\eta$ and $\overline{\sigma+t\eta}$ and by $NS_t$ the intersection $\Lambda^{1,1}_t\cap \Lambda_\Z$.

As a side product of our proof of Theorem \ref{theorem deformations are kahler replacement}, we will obtain a neat description of the preimage of $L$ under the period map.

\begin{theorem}
\label{theorem neat period map}
    Let $\pi\colon X\to B$ be a Lagrangian fibration on a $\Lambda$-marked irreducible hyperk\"ahler manifold $(X,f)$ and $L\subset\Omega_\Lambda$ its Shafarevich--Tate line. Denote by $\mathcal M_{\Sha T}$ the preimage of $L$ in $\mathcal M^0_\Lambda$ under the period map. Then the following holds.
    \begin{enumerate}
        \item The Hausdorff reduction $\overline{\mathcal M_{\Sha T}}$ of $\mathcal M_{\Sha T}$ is the disjoint union of copies of $L$.
        \item Every family $\mathcal Y\to T$ over an open subset $T\subset \mathcal M_{\Sha T}$ is a Shafarevich--Tate family outside possibly a closed subset contained in the preimage of $\{t\mid NS_t\not\subset \eta^\perp\}\subset L$. Moreover, the base of the Lagrangian fibrations in the Shafarevich--Tate family is isomorphic to $B$. 
        \item All points of $\mathcal M_{\Sha T}$ contained in the preimage of $\{t\mid NS_t\subset \eta^\perp\}\subset L$ are separated.
    \end{enumerate}
    
\end{theorem}

Part (1) of the theorem requires more work and will be showed in (\ref{subsub chamber decomposition}), but we can already show the rest and a weaker version of part (1).

\begin{lemma}
\label{lemma open dense}
In the assumptions and notation of Theorem \ref{theorem neat period map}, every connected component of the Hausdorff reduction $\overline{\mathcal M_{\Sha T}}$ of $\mathcal M_{\Sha T}$ is a dense open subset of $L$ whose complement is contained in $\{t\mid NS_t\subset \eta^\perp\}$.
\end{lemma}

\begin{proof}
    By Theorem \ref{theorem kahler substitute} and Corollary \ref{corollary all kahler except nowhere dense}, the statement of the lemma holds if a connected component $M$ of $\mathcal M_{\Sha T}$ has a hyperk\"ahler manifold $Y$ with a Lagrangian fibration $\pi\colon Y\to B$. Indeed, K\"ahler Shafarevich-Tate deformations of $Y$ map to an open dense subset of $L$. Moreover, Shafarevich--Tate deformations with $NS(Y_t)\not\subset\eta^\perp$ are projective, hence correspond to points in $\mathcal M_{\Sha T}$. 
    
    By Proposition \ref{proposition birational lagrangian fibrations}, every connected component of $\mathcal M_{\Sha T}$ contains an irreducible hyperk\"ahler manifold $Y_t$ in $\mathcal M_{\Sha T}$ with a Lagrangian fibration birational to the Lagrangian fibration on $X_t$. The claim follows.
\end{proof}

\begin{proof}[Proof of \ref{theorem neat period map}(2) and (3)]
Part (2) is a reformulation of Proposition \ref{proposition almost shafarevich tate family}. The statement about the bases is Proposition \ref{proposition bases are isomorphic}. Part (3) follows immediately from (2). Indeed, suppose $\mathcal Y\to \mathbb D$ and $\mathcal Y'\to \mathbb D$ are two families of hyperk\"ahler manifolds with periods in $L$ that coincide over an open subset $U\subset \mathbb D$ with $0\in \overline U$. Assume that $NS(Y_0)\subset \eta^\perp$. By Part (2), both families are Shafarevich--Tate families of a Lagrangian fibration on $Y_t$ for $t\in U$, at least over the subset where $NS(Y_t)\subset \eta^\perp$. Hence they coincide at $0$. 
\end{proof}

\subsubsection{Cones.} Let $(X,f)$ be a $\Lambda$-marked irreducible hyperk\"ahler manifold. Denote by $\Lambda^{1,1}_{\mathbb R}$ the image of $H^{1,1}(X,\R)$ in $\Lambda_\R$ under the marking $f$. Recall that the {\em positive cone} $\mathcal C$ of $X$ is defined to be the connected component of
$$
\{\alpha\in \Lambda^{1,1}_\R\mid q(\alpha)>0\}
$$
containing a K\"ahler class. Note that if $X$ admits a Lagrangian fibration $\pi\colon X\to B$ with isotropic class $\eta\in H^2(X,\Z)$, then $\mathcal C_X$ is contained in 
\begin{equation}
\overline{\mathcal C}:=\{\alpha\in \Lambda^{1,1}_\R\mid q(\alpha,\eta)>0\}.
\end{equation}
Similarly, we define
\begin{equation}
    \widetilde{\mathcal C} := \{\alpha\in \Lambda_\R\mid q(\alpha,\eta)>0\}.
\end{equation}

Let $Z\subset \Lambda^{1,1}_\R$ be a convex cone and $Z^\vee:=\{\alpha\in \Lambda_\R\mid q(\alpha,z)>0, \forall z\in Z\}$ its dual cone in $\Lambda_\R$. Define cones
\begin{gather}
\label{equation cones}
    \mathcal K(Z):= \mathcal C\cap Z^\vee;\\
    \overline{\mathcal K}(Z):= \overline{\mathcal C}\cap Z^\vee;\\
    \widetilde{\mathcal K}(Z):= \widetilde{\mathcal C}\cap Z^\vee.
\end{gather}

The following elementary lemma will be used in the proof of Theorem \ref{theorem neat period map}. The proof is straightforward and is left to the reader.

\begin{lemma}
\label{lemma cones}
    Let $Z$ be a convex cone contained in $\eta^\perp\cap \Lambda^{1,1}_\R$. Then the following holds.
    \begin{enumerate}
        \item $\overline{\mathcal K}(Z)  = \{\alpha\in \Lambda^{1,1}_\R\mid \exists t_0, \forall t>t_0, \alpha+t\eta\in \mathcal K(Z)\}$.
        \item $\widetilde{\mathcal K}(Z) = \overline{\mathcal K}(Z)+\langle \operatorname{Re}\sigma, \operatorname{Im}\sigma\rangle$.
        \item Let $Z'$ be another convex cone in $\eta^\perp\cap \Lambda^{1,1}_\R$. Then TFAE:
        $$
        \mathcal K(Z)\cap \mathcal K(Z')=\varnothing \Longleftrightarrow \overline{\mathcal K}(Z)\cap \overline{\mathcal K}(Z')=\varnothing \Longleftrightarrow \widetilde{\mathcal K}(Z)\cap \widetilde{\mathcal K}(Z')=\varnothing.
        $$
        \item Let $(Z_i)_{i\in I}$ be a collection of convex cones in $\eta^\perp\cap \Lambda^{1,1}_\R$. Then TFAE:
        \begin{multline*}
        \bigcup\limits_{i\in I}\mathcal K(Z_i) \text{ is an open dense subset of }\mathcal K(Z)\Longleftrightarrow\\ \bigcup\limits_{i\in I}\overline{\mathcal K}(Z_i) \text{ is an open dense subset of }\overline{\mathcal K}(Z)\Longleftrightarrow \\\bigcup\limits_{i\in I}\widetilde{\mathcal K}(Z_i) \text{ is an open dense subset of }\widetilde{\mathcal K}(Z).
        \end{multline*}
    \end{enumerate}
\end{lemma}

\subsubsection{Chamber decomposition of the positive cone.} 
\label{subsub chamber decomposition} The K\"ahler cone $\mathcal K_{(X,f)}\subset \Lambda_\R$ of a $\Lambda$-marked irreducible hyperk\"ahler manifold $(X, f)$ coincides with the cone
$$
\mathcal K_{(X,f)} = \mathcal K(f(\overline{NE}_X)),
$$
where $\overline{NE}_X\subset H^{1,1}(X,\R)$ is the closure of the cone of effective curves \cite{boucksom2001cone}. 

\begin{theorem}[{\cite[Theorem 1.4]{markman2011survey}}]
\label{theorem chamber decomposition}
    Let $(X,f)$ and $(Y,g)$ be two non-isomorphic $\Lambda$-marked irreducible hyperk\"ahler manifolds $\mathcal M^0_\Lambda$ with the same period $p\in \Omega_\Lambda$. Then the cones $\mathcal K_{(X,f)}$ and $\mathcal K_{(Y,g)}$ are disjoint. Moreover, the union of cones $\mathcal K_{(X,f)}$ over all points $(X,f)\in \mathcal M^0_\Lambda$ in the fiber $\operatorname{Per}^{-1}(p)$ is an open dense subset of the positive cone $\mathcal C$.
\end{theorem}





Let $\eta\in \Lambda^{1,1}$ be an isotropic class such that $q(\alpha,\eta)>0$ for some K\"ahler class $\alpha$ on $X$ (for example, the class of a Lagrangian fibration on $X$). As in (\ref{equation cones}) define
\begin{equation}
    \widetilde{\mathcal K}_{(X,f)}:= \widetilde{\mathcal K}(f(\overline{NE}_X))
\end{equation}

As a direct consequence of Theorem \ref{theorem chamber decomposition} and Lemma \ref{lemma cones} we obtain

\begin{lemma}
\label{lemma chamber decomposition}
    Let $(X,f)$ and $(Y,g)$ be two $\Lambda$-marked irreducible hyperk\"ahler manifolds $\mathcal M^0_\Lambda$ with the same period $p\in \Omega_\Lambda$. Assume that the Neron-Severi group of $X$ is contained in $\eta^\perp$ for an isotropic class $\eta$. Then the cones $\widetilde{\mathcal K}_{(X,f)}$ and $\widetilde{\mathcal K}_{(Y,g)}$ are disjoint. Moreover, the union of cones $\widetilde{\mathcal K}_{(X,f)}$ over all points $(X,f)\in \mathcal M^0_\Lambda$ in the fiber $Per^{-1}(p)$ is an open dense subset of the half-space $\widetilde{\mathcal C} = \{\alpha\mid q(\alpha,\eta)>0\}$.
\end{lemma}

\begin{lemma}
\label{lemma kahler cone}
    In the assumptions and notation of Theorem \ref{theorem neat period map}, let $(Y,g)$ be a $\Lambda$-marked irreducible hyperk\"ahler manifold in $\mathcal M_{\Sha T}$. Assume that $NS(Y)\subset \eta^\perp$. Then the following holds.
    \begin{enumerate}
        \item The cone $g(\overline {NE}_Y)$ is contained in $\eta^\perp\cap \Lambda^{1,1}_\R$.
        \item Let $(Y',g')$ be another $\Lambda$-marked irreducible hyperk\"ahler manifold with $NS(Y')\subset \eta^\perp$ in the same connected component of $\mathcal M_{\Sha T}$ as $(Y,g)$. Then the cones $\widetilde{\mathcal K}_{(Y,g)}$ and $\widetilde{\mathcal K}_{(Y',g')}$ coincide.
    \end{enumerate}
\end{lemma}

\begin{proof}
    Part (1) holds because $\overline{NE}_Y$ is a subset of $NS(Y)\otimes\R$. By Theorem \ref{theorem neat period map}(2), the manifolds $Y'$ and $Y$ admit Lagrangian fibrations and are Shafarevich--Tate twists of each other. All curves on $Y$ and $Y'$ are contained in the fibers of these Lagrangian fibrations by Lemma \ref{lemma eta perp}. Shafarevich--Tate twists of Lagrangian fibrations don't change fibers, hence the cones $g(\overline{NE}_Y)$ and $g'(\overline{NE}_{Y'})$ coincide. The claim follows.
\end{proof}

Let $M$ be a connected component of $\mathcal M_{\Sha T}$. Define the cone $\widetilde{\mathcal K}_M$ to be $\widetilde{\mathcal K}_{(X,f)}$ for any $\Lambda$-marked irreducible hyperk\"ahler manifold $(X,f)$ in $M$ with $NS(X)\subset \eta^\perp$. By Lemma \ref{lemma kahler cone} the definition doesn't depend on the choice of $(X,f)$.




\subsubsection{Proof of Theorem \ref{theorem neat period map}.}
\begin{proof}
    Let $t,s$ be two points on the Shafarevich--Tate line $L$. Assume that N\'eron-Severi groups $NS_t$ and $NS_s$ are contained in $\eta^\perp$. By Lemma \ref{lemma chamber decomposition}, the following cones 
    $$
    \widetilde{\mathcal C}_t := \bigsqcup\limits_{(X,f)\text{ s.t.} \operatorname{Per}(X,f)=t} \widetilde{\mathcal K}_{(X,f)}
    $$
    and
    $$
    \widetilde{\mathcal C}_s:= \bigsqcup\limits_{(X,f)\text{ s.t.} \operatorname{Per}(X,f)=s} \widetilde{\mathcal K}_{(X,f)}
    $$
    are open dense subsets in the same cone $\widetilde C$. Thanks to Lemma \ref{lemma kahler cone},
    \begin{gather*}
        \widetilde{\mathcal C}_t = \bigsqcup_{\substack{M \text{ conn comp of }\mathcal M_{\Sha T}, \\ t\in \operatorname{Per}(M)}} \widetilde{\mathcal K}_M\\
        \widetilde{\mathcal C}_s = \bigsqcup_{\substack{M \text{ conn comp of }\mathcal M_{\Sha T}, \\ s\in \operatorname{Per}(M)}} \widetilde{\mathcal K}_M\\
    \end{gather*}
    For any connected components $M$ and $M'$ of $\mathcal M_{\Sha T}$ the cones $\mathcal K_M$ and $\mathcal K_{M'}$ are disjoint. The closure of both $\widetilde{\mathcal C}_t$ and $\widetilde{\mathcal C}_s$ is $\tilde C$ (Lemma \ref{lemma chamber decomposition}). Therefore, for any two points $t, s\in L$ such that $NS_t$ and $NS_s$ are contained in $\eta^\perp$, the set of components $M$ such that $t\in \operatorname{Per}(M)$, must coincide with the set of components such that $s\in \operatorname{Per}(M)$. By Lemma \ref{lemma open dense}, the image of any component $M$ under the period map is an open dense subset of $L$, hence it contains a point $t$ with $NS_t\subset\eta^\perp$. It follows that the image of $M$  contains the entire subset of $t\in L$ such that $NS_t\in \eta^\perp$. By Lemma \ref{lemma open dense} it also contains the subset of $t$ such that $NS_t\not\in\eta^\perp$, hence the map $M\to L$ is surjective. Since it is a local isomorphism, the Hausdorff reduction $\overline M$ maps isomorphically onto $L$.
    
\end{proof}



\subsubsection{Proof of Theorem \ref{theorem deformations are kahler replacement}}
\label{subsub proof of theorem kahler}
\begin{proof}
    The theorem follows from Theorem \ref{theorem neat period map}. Choose a $\Lambda$-marking on the family $\mathcal X$. Consider the connected component $\mathcal M_{\Sha T}^\circ$ of $\mathcal M_{\Sha T}$ containing hyperk\"ahler manifolds $X_t$ for $t\in U$. Since $X_0$ is non-projective, we have $NS(X_0)\subset \eta^\perp$. By Theorem \ref{theorem neat period map}(1) the preimage of $0$ in $\mathcal M_{\Sha T}^\circ$ is non-empty and by \ref{theorem neat period map}(3) it consists of one point corresponding to a hyperk\"ahler manifold $X_0'$. By Theorem \ref{theorem neat period map}(2), the class $\eta$ induces a Lagrangian fibration on $X_0'$. Consider the Shafarevich--Tate family $\mathcal X'$ of $X_0'$. By Theorem \ref{theorem neat period map}(2), the families $\mathcal X'$ and $\mathcal X$ coincide over the set of points of $U$ where $NS(X_t)= NS(X'_t)\subset \eta^\perp$. Local Torelli theorem implies that the families $\mathcal X'$ and $\mathcal X$ coincide over an open subset of $U$. They are both Shafarevich--Tate families, hence they coincide everywhere, in particular, $X_0\simeq X_0'$. The manifold $X'_0$ is K\"ahler, hence so is $X_0$.
\end{proof}




\section{Topology of Shafarevich--Tate twists}
\label{section topology of shafarevich}

In this section we will prove Theorems \ref{theor irr hol symp}, \ref{theorem b2} and \ref{theorem fujiki class c}.

\subsubsection{Higher pushforwards of $\Q_X$ do not depend on a twist.} Let $\pi\colon X\to B$ be a Lagrangian fibration and $\pi^\phi\colon X^\phi\to B$ its Shafarevich--Tate twist. Then the sheaves $R^k\pi_*\Z$ and $R^k\pi^\phi_*\Z$ are canonically identified. Indeed, represent $\phi$ as a \v{C}ech cocycle $(\phi_{ij})$, where $\phi_{ij}\in Aut^0_{X/B}(U_{ij})$. The automorphisms $\phi_{ij}$ are flows of vector fields, hence they act trivially on $H^k(X_{ij})$.

In particular the vector spaces $H^0(R^2\pi_*\Q)$ and $H^0(R^2\pi^\phi_*\Q)$ are canonically identified. However, the differentals
$$
d_2\colon H^0(R^2\pi_*\Q) \to H^2(R^1\pi_*\Q)\:\:\:\:\:\text{and}\:\:\:\:\: d_2^\phi\colon H^0(R^2\pi_*\Q) \to H^2(R^1\pi_*\Q)
$$
from the Leray spectral sequence of $X$ and $X^\phi$ respectively may be different.

\subsubsection{The restriction map $H^2(X)\to H^2(F)$ has rank at most one.}\label{subsub restriction map} Suppose that $X$ is hyperk\"ahler, and let $F$ be a smooth fiber of $\pi\colon X\to B$. By Theorem \ref{theorem_restriction}, the restriction map $H^2(X) \to H^2(F)$ has a one-dimensional image generated by an ample class. Global invariant cycle theorem implies that 
$$
\im \left(H^2(X,\Q)\to H^2(F,\Q)\right) = H^2(F,\Q)^{\pi_1(B^\circ)} = H^0(B^\circ, R^2\pi_*\Q\restrict{B^\circ}).
$$
Here $H^2(F,\Q)^{\pi_1(B^\circ)}$ denotes the subspace of $H^2(F)$ invariant under the monodromy action of $\pi_1(B^\circ)$. It follows that $H^2(F,\Q)^{\pi_1(B^\circ)}$ is one-dimensional and generated by an ample class.

Let $X^\phi$ be a Shafarevich--Tate twist of $X$, not necessarily K\"ahler. Then the image of the map $H^2(X^\phi,\Q)\to H^2(F,\Q)$ still lies in $H^2(F,\Q)^{\pi_1(B^\circ)}$. The latter space is isomorphic to $H^0(B^\circ,R^2\pi_*\Q\restrict{B^\circ})$, hence does not depend on a twist. We obtain the following statement.

\begin{proposition}
\label{proposition image of restriction map}
    Let $\pi\colon X\to B$ be a Lagrangian fibration on an irreducible hyperk\"ahler manifold $X$ and $X^\phi$ be its Shafarevich--Tate twist. Then the restriction map
    $$
    H^2(X^\phi)\to H^2(F)
    $$
    is either trivial or has a one-dimensional image generated by an ample class of $F$.
\end{proposition}


\subsection{First cohomology of twists}
\label{subsection fundamental group of twists}

\begin{lemma}
\label{lemma base is simply connected}
    Let $\pi\colon X\to B$ be a Lagrangian fibration on an irreducible hyperk\"ahler manifold. Then $B$ is simply connected.
\end{lemma}
\begin{proof}
    For some $\psi\in\Sha^0$, the twist $X^\psi$ is projective. Hence we may and will assume that $X$ is projective. If $f\colon M\to N$ is a dominant map of normal algebraic varieties such that the general fiber of $f$ is irreducible, then $f(\pi_1(M)) = \pi_1(N)$ \cite[Proposition 2.10.2]{kollar1995shafarevich}. Therefore, $\pi_1(B) = \pi_1(X) = 0$. 
\end{proof}

\begin{proposition}
\label{proposition H1 vanishes}   
    Let $\pi\colon X\to B$ be a Lagrangian fibration on an irreducible hyperk\"ahler manifold $X$ and $X^\phi$ its Shafarevich--Tate twist. Then $H^1(X^\phi,\Q) = 0$. 
\end{proposition}
\begin{proof}
For any Lagrangian fibration $\pi\colon X\to B$ on a hyperk\"ahler manifold, the pullback map $H^2(B,\Q)\to H^2(X,\Q)$ is injective \cite[Corollary 1.13]{huybrechts2022lagrangian}. It follows from Leray spectral sequence that the sequence
\begin{equation}
\label{seq leray h1}
0\to H^1(B,\Q)\to H^1(X,\Q)\to H^0(B,R^1\pi_*\Q)\to 0
\end{equation}
is exact. Since $B$ and $X$ are simply connected (Lemma \ref{lemma base is simply connected}), the group $H^0(B,R^1\pi_*\Q)$ vanishes. The exact sequence (\ref{seq leray h1}) for $X^\phi$ implies that for any Shafarevich--Tate twist $X^\phi$,
$$
    H^1(X^\phi,\Q) \simeq H^1(B,\Q)=0.
$$
\end{proof}

\subsection{Hodge numbers of twists}
\label{subsection hodge numbers of twists}
Recall that by \cite[Corollary 3.7]{abasheva2025shafarevich}, a Shafarevich--Tate twist $X^\phi$ inherits a holomorphic symplectic form $\sigma_\phi$. Namely, one can show that any class $\phi\in \Sha$ can be represented by a \v{C}ech cocycle $\phi_{ij}\in Aut^0_{X/B}(U_{ij})$ such that $\phi_{ij}$ preserves the holomorphic symplectic form $\sigma$ on $X$. The holomorphic symplectic form $\sigma_\phi$ is obtained by patching the forms $\sigma|_{X_i}$ using the automorphisms $\phi_{ij}$.
\begin{proposition}
\label{proposition_cohomology_of_ox}
    Let $\pi\colon X\to B$ be a Lagrangian fibration on an irreducible hyperk\"ahler manifold $X$. Then $\forall \phi\in\Sha$:
    $$
        H^{0,k}(X^\phi) := H^k(X^\phi,\mathcal O_{X^\phi}) = \begin{cases}
            0,\text{   if }k\text{ is odd;}\\
            \C\cdot\overline{\sigma_\phi}^{k/2},\text{   if }k\text{ is even.}
        \end{cases}
    $$
\end{proposition}
\begin{proof}
    For any $\phi\in\Sha$, the sheaf $R^i\pi^\phi_*\mathcal O_{X^\phi}$ is isomorphic to $\Omega^{[i]}_B$ by Corollary \ref{corollary matsushita}. The Leray spectral sequence for $\mathcal O_{X^\phi}$ has the form
    \begin{equation}
    \label{spectral seq}
    E_2^{p,q} = H^q(R^p\pi_*\mathcal O_{X^\phi}) \simeq H^q(\Omega^{[p]}_B) = \begin{cases}
        0,\text{   if }p\ne q;\\
        \C,\text{   otherwise.}
    \end{cases}
    \end{equation}
    This computation follows from Theorem \ref{theorem cohomology of omega}. The spectral sequence (\ref{spectral seq}) degenerates at $E_2$, hence $H^{0,k}(X^\phi) = 0$ for $k$ odd and $H^{0,k}(X^\phi) = H^{k/2,k/2}(B)=\C$ when $k$ is even. 
    
    The cohomology group $H^{0,2r}(X^\phi)$ is generated by the class of $\overline{\sigma_\phi}^r$. Indeed, the form $\overline{\sigma_\phi}^r$ is $d$-closed and not $\bar\partial$-exact, because if $\overline{\sigma_\phi}^r = \bar\partial \alpha$, then
    $$
    0 = \int_X\overline\partial(\alpha\overline{\sigma_\phi}^{(n-r)}{\sigma_\phi}^n) = \int_Xd(\alpha\overline{\sigma_\phi}^{(n-r)}\sigma_\phi^n) = \int_X\overline{\sigma_\phi}^n\sigma_\phi^n \ne 0,
    $$
    contradiction. 
\end{proof}

Next we will compute $H^0(\Omega^2_{X^\phi})$ for a Shafarevich-Tate twist $X^\phi$. We will start with a few preliminary lemmas.

\begin{lemma}
\label{lemma holomorphic 2 form restricts trivially}
    Let $\xi$ be a holomorphic $2$-form on $X^\phi$. Then $\xi$ restricts trivially to all smooth fibers.
\end{lemma}
\begin{proof}
    The restriction of $\xi$ to every smooth fiber is $d$-closed because all holomorphic forms on K\"ahler manifolds are closed. Therefore, $\xi$ defines a section of the local system $R^2\pi_*\C\restrict{B^\circ}$. By \ref{subsub restriction map}, this local system has just one non-trivial section, which is the class of a form of type $(1,1)$. The class $[\xi\restrict{F}]$ is of type $(2,0)$, hence it must be trivial. There are no non-trivial exact holomorphic forms on $F$, hence $\xi\restrict{F} = 0$ for every smooth fiber.
\end{proof}

\begin{lemma}
\label{lemma contraction with holomorphic 2 forms}
    Let $\pi\colon Y\to S$ be a proper Lagrangian fibration over a not necessarily compact base. Consider a holomorphic 2-form $\xi$ on $Y$ with trivial restriction to every smooth fiber. Then $\xi$ induces a map
    $$
        \iota_\xi\colon \pi_*T_{Y/S} \to \Omega^{[1]}_S.
    $$
\end{lemma}
\begin{proof}
Consider the map 
    $$
    \iota_\xi\colon T_{Y/S}\to \Omega_Y
    $$
    sending a vector field $v$ to $\iota_v\xi$. As before, denote by $Y^\circ$ the union of smooth fibers of $\pi$ and $S^\circ:=\pi(Y^\circ)$. For every vertical vector field $v$, the restriction of $\iota_v\xi$ to $Y^\circ$ lies in $\pi^*\Omega_{S^\circ}$ because $\xi\restrict{F} = 0$ for every smooth fiber $F$. It follows that the image of $\iota_\xi$ lies in the sheaf $(\pi^*\Omega_S)^{sat}$ consisting of $1$-forms $\alpha$ such that $\alpha\restrict{Y^\circ}\in\pi^*\Omega_S^\circ$. By taking pushforwards, we obtain a map
    $$
    \iota_\xi\colon \pi_*T_{Y/S} \to \pi_*(\pi^*\Omega^1_S)^{sat}.
    $$
    We will show that $\pi_*(\pi^*\Omega_S)^{sat}\simeq \Omega_S^{[1]}$. Indeed, this is definitely true over $S^\circ$. Let $\alpha$ be a local section of $(\pi^*\Omega_S)^{sat}$. Then the restriction of $\alpha$ to $Y^\circ$ is the pullback of a form from $S^\circ$. By Lemma \ref{lemma pullbacks of forms}, the form $\alpha$ must be the pullback of a reflexive form from $S$. 
\end{proof}

\subsubsection{}Let $\pi\colon X\to B$ be a Lagrangian fibration. Consider the subsheaf $(\pi_*\Omega^2_X)'$ of $\pi_*\Omega^2_X$ consisting of holomorphic 2-forms $\xi$ with trivial restriction to all smooth fibers. Thanks to Lemma \ref{lemma contraction with holomorphic 2 forms}, there is a natural map
\begin{equation}
\label{map contraction}
(\pi_*\Omega^2_X)'\to \mathcal{Hom}(\pi_*T_{X/B},\Omega^{[1]}_B).
\end{equation}
The holomorphic symplectic form $\sigma$ on $X$ induces an isomorphism $\pi_*T_{X/B}\simeq \Omega^{[1]}_B$ (Theorem \ref{theorem iso omega and t}). Composing the map (\ref{map contraction}) with this isomorphism, we obtain a map of sheaves
$$
    \rho \colon (\pi_*\Omega^2_X)'\to \mathcal{End}(\Omega^{[1]}_B).
$$

\begin{lemma}
\label{lemma holomorphic 2 form end doesn't change}
    Define the sheaf $\mathcal{End}'_X(\Omega^{[1]}_B)$ as the image of $\rho$. Then for any Shafarevich--Tate twist $X^\phi$ the sheaf $\mathcal{End}'_{X^\phi}(\Omega^{[1]}_B)$ coincides with $\mathcal{End}'_X(\Omega^{[1]}_B)$.
\end{lemma}
\begin{proof}
    The statement is local on $B$. For every sufficiently small open disk $U\subset B$, the manifolds $\pi^{-1}(U)$ and $(\pi^\phi)^{-1}(U)$ are isomorphic as Lagrangian fibrations, hence the claim.
\end{proof}

\begin{lemma}
\label{lemma holomorpic 2 form ex sequence}
    The sequence of sheaves on $B$
    $$
        0\to \Omega^{[2]}_B\to (\pi_*\Omega^2_X)' \to \mathcal{End}'_X(\Omega^{[1]}_B) \to 0
    $$
    is exact.
\end{lemma}
\begin{proof}
    Note that the map $(\pi_*\Omega^2_X)' \to \mathcal{End}'_X(\Omega^{[1]}_B)$ is surjective by the definition of $\mathcal{End}'_X(\Omega^{[1]}_B)$. The first map $\Omega^{[2]}_B\to (\pi_*\Omega^2_X)'$ is clearly injective. 
    
    The composite map $\Omega^{[2]}_B\to \mathcal{End}'_X(\Omega^{[1]}_B)$ vanishes. Indeed, let $\alpha$ be a local section of $\Omega^{[2]}_B$. Then for any vertical vector field $v$, the form $\iota_v\pi^*\alpha$ vanishes on $X^\circ$, hence vanishes everywhere. Therefore, $\rho(\alpha) = 0$.

    It remains to prove exactness in the middle term. Let $U\subset B$ be an open subset and $\xi$ a holomorphic $2$-form on $\pi^{-1}(U)$ such that $\rho(\xi) = 0$. Consider the restriction of $\xi$ to $X^\circ$. Since $\iota_v\xi = 0$ for every vertical vector field $v$, the form $\xi$ is contained in $\pi^*\Omega^2_{B}(\pi^{-1}(U\cap B^\circ))$. The projection formula together with the fact that $\pi_*\mathcal O_{X^\circ} \simeq \mathcal O_{B^\circ}$ implies that 
    $$
    \pi^*\Omega^2_{B}(\pi^{-1}(U\cap B^\circ)) = \pi_*\pi^*\Omega^2_B(U\cap B^\circ) =\Omega^2_{B}(U\cap B^\circ).
    $$ 
    Hence there exists a holomorphic $2$-form $\alpha^\circ$ on $U\cap B^\circ$ such that $\xi\restrict{\pi^{-1}(U\cap B^\circ)} = \pi^*\alpha^\circ$. By Lemma \ref{lemma pullbacks of forms}, $\xi = \pi^*\alpha$ for some reflexive holomorphic $2$-form $\alpha$ on $U$. 
\end{proof}

We are finally ready to show that all holomorphic $2$-forms on $X^\phi$ are multiples of $\sigma_\phi$.
\begin{theorem}
\label{theorem generated by sigma}
Let $\pi\colon X\to B$ be a Lagrangian fibration on an irreducible hyperk\"ahler manifold $X$. Then $H^0(\Omega^2_{X^\phi})$ is generated by the holomorphic symplectic form $\sigma_\phi$ for all $\phi\in \Sha$.
\end{theorem}
\begin{proof}
    By Lemma \ref{lemma holomorphic 2 form restricts trivially}, every holomorphic $2$-form $\xi$ on $X^\phi$ restricts trivially to every smooth fiber. Therefore,
    $$
    H^0(B,(\pi_*\Omega^2_{X^\phi})') = H^0(X^\phi,\Omega^2_{X^\phi}).
    $$
    Lemma \ref{lemma holomorpic 2 form ex sequence} shows that the sequence
    $$
    0\to \Omega^{[2]}_B\to (\pi_*\Omega^2_{X^\phi})' \to \mathcal{End}'_{X^\phi}(\Omega^{[1]}_B)\to 0
    $$
    is exact. Consider its long exact sequence of cohomology
    $$
    0 \to H^0(\Omega^{[2]}_B) \to H^{2,0}(X^\phi)\to H^0(\mathcal{End}'_{X^\phi}(\Omega^{[1]}_B))\to H^1(\Omega^{[2]}_B)
    $$
    The cohomology groups $H^i(\Omega^{[2]}_B)$ vanish for $i=0,1$ (Theorem \ref{theorem cohomology of omega}). Therefore,
    $$
    H^{2,0}(X^\phi)\simeq H^0(\mathcal{End}'_{X^\phi}(\Omega^{[1]}_B)).
    $$
    The sheaf $\mathcal{End}_{X^\phi}'(\Omega^{[1]}_B)$ does not depend on a twist by Lemma \ref{lemma holomorphic 2 form end doesn't change}, therefore $H^{2,0}(X^\phi)$ does not depend on a twist.
\end{proof}

\begin{remark}
    When $B = \mathbb P^n$, the proof of Theorem \ref{theorem generated by sigma} can be simplified because $\operatorname{End}(\Omega^1_{\mathbb P^n})\simeq \C$. By Lemma \ref{lemma holomorphic 2 form restricts trivially}, every holomorphic $2$-form $\xi$ on $X^\phi$ restricts trivially to smooth fibers. Hence $\xi$ induces an endomorphism $\rho(\xi)$ of $\Omega^1_{\mathbb P^n}$ (Lemma \ref{lemma contraction with holomorphic 2 forms}). Since $\operatorname{End}(\Omega^1_{\mathbb P^n}) = \C$, there exists a number $\lambda\in\C$ such that $\rho(\xi - \lambda\sigma) = 0$. The contraction of every vertical vector field on $X$ with $\xi-\lambda\sigma$ vanishes, hence
    $$
    (\xi - \lambda\sigma)\restrict{X^\circ}\in\pi^*\Omega^2_{(\mathbb P^n)^\circ}.
    $$
    Since $\pi_*\pi^*\Omega^2_{\mathbb P^n} = \Omega^2_{\mathbb P^n}$, we have $(\xi - \lambda\sigma)\restrict{X^\circ} = \pi^*\alpha^\circ$ for some holomorphic $2$-form $\alpha^\circ$ on $(\mathbb P^n)^\circ$. By Lemma \ref{lemma pullbacks of forms}, $\alpha^\circ$ extends to a holomorphic form $\alpha$ on $\mathbb P^n$ and $\xi - \lambda\sigma = \pi^*\alpha$. There are no non-trivial holomorphic forms on $\mathbb P^n$, hence $\xi = \lambda\sigma$.
    
    We were unable to show that $\operatorname{End}(\Omega^{[1]}_B)\simeq \C$ for any base of a Lagrangian fibration, although we expect it to be true.
\end{remark}

\subsubsection{Proof of Theorem \ref{theor irr hol symp}.}\label{subsub proof of theorem b}The statement immediately follows from Proposition \ref{proposition H1 vanishes} and Theorem \ref{theorem generated by sigma}.\hfill$\qed$


\subsection{Second cohomology of a twist}
\label{subsection second cohomology of a twist}

Our goal now is to prove Theorems \ref{theorem b2} and Theorem \ref{theorem fujiki class c}.

\begin{lemma}
\label{lemma global sections of ns}
    Let $\pi\colon X\to B$ be a Lagrangian fibration on an irreducible hyperk\"ahler manifold $X$. Define the sheaf $\mathcal{NS}$ on $B$ as the image of the Chern class map $R^1\pi_*\mathcal O_X^\times\to R^2\pi_*\Z$. Then
    $$
    H^0(B,\mathcal{NS}) = H^0(R^2\pi_*\Z).
    $$
    In other words, for every section $\xi$ of $R^2\pi_*\Z$ and a sufficiently fine open cover $B = \bigcup U_i$, there are line bundles $L_i$ on $X_i$ such that $\xi\restrict{U_i}=c_1(L_i)$. In particular, every section $\xi$ of $R^1\pi_*\Z$ is locally the class of a closed $(1,1)$-form.
\end{lemma}
\begin{proof}
    Consider the exponential exact sequence 
    $$
    0\to \Z_X\to \mathcal O_X\to \mathcal O_X^\times\to 0.
    $$
    It induces a long exact sequence of pushforward sheaves:
    $$
    R^1\pi_*\mathcal O_X^\times\to R^2\pi_*\Z\to R^2\pi_*\mathcal O_X.
    $$
    The sheaf $R^2\pi_*\Z/\im (R^1\pi_*\mathcal O_X^\times) = R^2\pi_*\Z/\mathcal{NS}$ is a subsheaf of $R^2\pi_*\mathcal O_X\simeq \Omega^{[2]}_B$. Since $H^0(\Omega^2_B) = 0$ (Theorem \ref{theorem cohomology of omega}), the sheaf $R^2\pi_*\Z/\mathcal{NS}$ has no global sections. Hence the natural inclusion $H^0(B,\mathcal{NS})\to H^0(R^2\pi_*\Z)$ is an isomorphism.
\end{proof}

\subsubsection{Isomorphisms between $T_{X/B}$ and $R^1\pi_*\mathcal O_X$.}\label{subsub iso pistart and ronepistarox} Let $\xi$ be a global section of $H^0(R^2\pi_*\Q)$. It defines a map $f_\xi\colon \pi_*T_{X/B}\to R^1\pi_*\mathcal O_X$ in a similar way that a class $\omega\in H^2(X,\Q)$ defines a map $f_\omega$ in \ref{subsub higher pushforwards}. Namely, by Lemma \ref{lemma global sections of ns} we can represent $\xi\restrict{U_i}$ by a closed $(1,1)$-form $\xi_i$ on $X_i$. Consider the map
$$
f_{\xi_i}\colon \pi_*T_{X_i/U_i}\to R^1\pi_*\O_X\restrict{U_i}
$$
sending $v$ to the class of $[\iota_v\xi_i]$ under the $\bar\partial$-differential. Since the sheaf $R^1\pi_*\mathcal O_X$ is torsion-free (Theorem \ref{theorem matsushita}), the map $f_{\xi_i}$ is determined uniquely by its restriction to $B^\circ\cap U_i$. The restrictions of both sheaves $\pi_*T_{X_i/U_i}$ and $R^1\pi_*\O_X\restrict{U_i}$ to $B^\circ\cap U_i$ are vector bundles. For every point $b\in B^\circ\cap U_i$, the map $f_{\xi_i}$ over $b$ is the map
$$
H^0(F, T_{F}) = H^{1,0}(F) ^\vee\to H^{0,1}(F)
$$
given by the contraction with $[\xi\restrict{F}]\in H^{1,1}(F)$. Here $F$ denotes $\pi^{-1}(b)$. Therefore, the map $f_{\xi_i}$ depends only on the class $[\xi\restrict{F}]$ of the restriction of $\xi$ to a smooth fiber $F$. In particular, maps $f_{\xi_i}$ do not depend on the choice of the forms $\xi_i$ representing $\xi\in H^0(R^2\pi_*\Q)$ and glue into a well-defined map
$$
f_\xi\colon \pi_*T_{X/B}\to R^1\pi_*\O_X.
$$

\subsubsection{}\label{subsub isos} The argument above also shows that the map $H^0(B, R^2\pi_*\C)\to \operatorname{Hom}(\pi_*T_{X/B},R^1\pi_*\mathcal O_X)$ sending $\xi$ to $f_\xi$ factors through the restriction to a smooth fiber $F$:
$$
    H^0(B, R^2\pi_*\C)\to H^0(B^\circ, R^2\pi_*\C) = H^2(F)^{\pi_1(B^\circ)} \to \operatorname{Hom}(\pi_*T_{X/B},R^1\pi_*\mathcal O_X).
$$
The vector space $H^2(F)^{\pi_1(B^\circ)}$ is one-dimensional by \ref{subsub restriction map}. Fix an element $\xi_0\in H^0(B, R^2\pi_*\Z)$ which restricts non-trivially to $F$ and let $f_0:=f_{\xi_0}$ be the induced isomorphism $\pi_*T_{X/B}\to R^1\pi_*\mathcal O_X$. It follows that for every $\xi\in H^0(B, R^2\pi_*\C)$ there exists a unique number $\lambda_\xi$ such that
\begin{equation}
\label{equation_lambda}
    f_{\xi} = \lambda_{\xi} f_0.
\end{equation}
Recall that the isomorphism $f_0\colon \pi_*T_{X/B}\to R^1\pi_*\mathcal O_X$
sends $\Gamma_\Q=\ker(\pi_*T_{X/B}\to Aut^0_{X/B})\otimes\Q \subset \pi_*T_{X/B}$ isomorphically onto $R^1\pi_*\Q\subset R^1\pi_*\mathcal O_X$ (\ref{subsub gamma and r1}). We identify the group $H^2(\Gamma_\Q) = (\Sha/\Sha^0)\otimes \Q$ with $H^2(R^1\pi_*\Q)$ using the isomorphism $f_0\restrict{\Gamma_\Q}\colon \Gamma_\Q\to R^1\pi_*\Q$.


\subsubsection{Boundary map $\Sha\to H^2(\Gamma)$.} 
\label{subsub boundary for sha}The boundary map $\Sha=H^1(Aut^0(X/B)) \to H^2(\Gamma)$ coming from the short exact sequence
$$
0\to\Gamma\to\pi_*T_{X/B}\to Aut^0_{X/B}\to 0
$$
can be described in terms of \v{C}ech cocycles as follows. Pick $\phi\in \Sha$ and represent it by a $1$-cocycle $\phi_{ij}\in Aut^0_{X/B}(U_{ij})$. We can find a vertical vector field $v_{ij}$ on $X_{ij}$ such that $\exp(v_{ij}) = \phi_{ij}$. The vector field $v_{ij} + v_{jk} + v_{ki}$ on $X_{ijk}$ lies in $\Gamma$ thanks to the cocycle condition on $\phi_{ij}$. It represents the class $\overline\phi\in H^2(\Gamma)$, where $\overline\phi$ denotes the image of $\phi$ under the boundary map $\Sha\to H^2(\Gamma)$.

\subsubsection{Boundary map $H^0(R^2\pi_*\Q)\to H^2(R^1\pi_*\Q)$.}\label{subsub boundary for forms} We will describe the boundary map 
$$
d_2\colon H^0(R^2\pi_*\Q)\to H^2(R^1\pi_*\Q) \simeq H^2(\Gamma_\Q)
$$ 
from the Leray spectral sequence of $\pi$ in terms of \v{C}ech cocycles. Let $\xi$ be a section of $H^0(R^2\pi_*\Q)$. Represent it locally by $(1,1)$-forms $\xi_i$ on $X_i$. The difference $\xi_j - \xi_i$ is an exact form, hence 
$$
\xi_j - \xi_i = d\rho_{ij}
$$
for some $1$-form $\rho_{ij}$ on $X_{ij}$. The form $\rho_{ij} + \rho_{jk} + \rho_{ki}$ is closed on $X_{ijk}$, hence defines a cocycle with coefficients in $R^1\pi_*\Q$. By \ref{subsub gamma and r1}, there exists a unique vertical vector field $w_{ijk}\in \Gamma_\Q(U_{ijk})$ on $X_{ijk}$ such that $f_0(w_{ijk})$ is equal to the class of the (0,1)-form $\rho^{0,1}_{ij} + \rho^{0,1}_{jk} + \rho^{0,1}_{ki}$ under the $\bar\partial$-differential. The class of the cocycle $\{w_{ijk}\}$ in $H^2(\Gamma_\Q)$ is the image of $\xi$ under the boundary map.

\begin{proposition}
\label{proposition d_2}
    Let $\pi\colon X\to B$ be a Lagrangian fibration. Pick a class $\phi\in \Sha$. Let 
     $$
     d_2^\phi\colon H^0(R^2\pi_*\Q)\to H^2(R^1\pi_*\Q)\simeq H^2(\Gamma_\Q)
     $$
     be the differential in the Leray spectral sequence for $\pi^\phi\colon X^\phi\to B$.  Then for any $\xi\in H^0(R^2\pi_*\Q)$
     $$
     d_2^\phi(\xi) = d_2(\xi) + \lambda_\xi\overline\phi,
     $$
     where $\overline\phi$ is the image of $\phi\in \Sha$ under the map $\Sha\to H^2(\Gamma)$ and $\lambda_\xi$ is as in (\ref{equation_lambda}).
\end{proposition}

\begin{proof}
Represent $\xi\in H^0(R^2\pi_*\Q)$ by a collection of closed $(1,1)$-forms $\xi_i$ on $X_i$. When we view $X_i$ as an open subset of $X^\phi$, we will denote the same forms by $\xi_i^\phi$. The difference $\xi_j^\phi - \xi_i^\phi$ is not the same as $\xi_j - \xi_i$ because there is a twist by $\phi_{ij}\in Aut^0_{X/B}(U_{ij})$ involved. Namely
    $$
    \xi_j^\phi - \xi_i^\phi = \phi_{ij}^*\xi_j - \xi_i = (\phi_{ij}^*\xi_j - \xi_j) + (\xi_j - \xi_i).
    $$
    As in \ref{subsub boundary for forms}, write $\xi_j - \xi_i = d\rho_{ij}$. Find a vector field $v_{ij}$ such that $\phi_{ij} = \exp(v_{ij})$. Then
    \begin{multline*}
        \phi_{ij}^*\xi_j - \xi_j = \int\limits_0^1\frac{d}{dt}\left(\exp(tv_{ij})^*\xi_j\right) dt = \int\limits_0^1\exp(tv_{ij})^*(\mathrm{L}_{v_{ij}}\xi_j) dt = \\
        =\int\limits_0^1\exp(tv_{ij})^*(d\iota_{v_{ij}}+\iota_{v_{ij}}d)\xi_j dt =
        d\int\limits_0^1\exp(tv_{ij})^*\iota_{v_{ij}}\xi_j dt.
    \end{multline*}
    Here $\mathrm{L}$ denotes the Lie derivative. The second equality follows from the definition of the Lie derivative:
    $$
    \mathrm{L}_{v}\xi = \left.\frac{d}{dt}(\exp(tv)^*\xi) \right|_{t=0}.
    $$
    The third equality is the Cartan formula 
    $$
    \mathrm{L} = d\iota_v + \iota_vd,
    $$
    and the last equality holds because $\xi_j$ is closed. Set $\gamma_{ij} := \int\limits_0^1\exp(tv_{ij})^*\iota_{v_{ij}}\xi_j dt$. Then
    \begin{equation}
    \label{equation_xi_rho_gamma}
    \xi_j^\phi - \xi_i^\phi = d(\rho_{ij} + \gamma_{ij}).
    \end{equation}
    The form $\gamma_{ij}$ is of type $(0,1)$ and $\bar\partial$-closed. Its class under the $\bar\partial$-differential is
    $$
        \int\limits_0^1[\exp(tv_{ij})^*\iota_{v_{ij}}\xi_j]dt = \int\limits_0^1[\iota_{v_{ij}}\xi_j]dt = [\iota_{v_{ij}}\xi_j] = f_\xi(v_{ij}) = \lambda_\xi f_0(v_{ij}).
    $$
    The third equality follows from the definition of $f_\xi$ in \ref{subsub iso pistart and ronepistarox} and the last equality from the definition of $\lambda_\xi$ in (\ref{equation_lambda}). As in \ref{subsub boundary for forms}, let $w_{ijk}$ be the unique  vector field in $\Gamma_\Q(U_{ijk})$ such that $f_0(w_{ijk}) = [\rho^{0,1}_{ij} + \rho^{0,1}_{jk} + \rho^{0,1}_{ki}]$. It follows from (\ref{equation_xi_rho_gamma}) that the class $d_2^\phi(\xi)$ can be represented by the cocycle
    $$
    w_{ijk} + \lambda_\xi(v_{ij} + v_{jk} + v_{ki}).
    $$
    Indeed,
    $$
    f_0(w_{ijk} + \lambda_\xi(v_{ij} + v_{jk} + v_{ki})) = [\rho^{0,1}_{ij} + \gamma_{ij} + \rho^{0,1}_{jk} + \gamma_{jk} + \rho^{0,1}_{ki} + \gamma_{ki}].
    $$
    By \ref{subsub boundary for forms}, the class of the cocycle $w_{ijk}$ in $H^2(\Gamma_\Q)$ is $d_2(\xi)$ and by \ref{subsub boundary for sha}, the class of the cocycle $v_{ij} + v_{jk} + v_{ki}$ in $H^2(\Gamma_\Q)$ is $\overline\phi$. The claim follows.
    
\end{proof}

\begin{corollary}
\label{corollary differentials coincide}
    Let $\pi\colon X\to B$ be a Lagrangian fibration. Then for any $\phi\in \Sha'$ we have $d_2^\phi = d_2$.
\end{corollary}
\begin{proof}
    For any $\phi\in\Sha'$, the class $\overline\phi\in H^2(\Gamma_\Q)$ vanishes by the definition of $\Sha'$. Proposition \ref{proposition d_2} implies that $d_2^\phi = d_2$.
\end{proof}

\begin{corollary}
\label{corollary kernels equal}
    Let $\pi\colon X\to B$ be a Lagrangian fibration. Pick a class $\phi\in \Sha$ such that $\overline\phi\in H^2(\Gamma_\Q)$ does not vanish. Consider the restriction maps
    $$
    r\colon H^2(X)\to H^0(R^2\pi_*\Q),\:\:\:\:\:r^\phi\colon H^2(X^\phi)\to H^0(R^2\pi_*\Q).
    $$
    Let $H^2(X)^{0}$ (resp. $H^2(X^\phi)^{0})$ denote the subspace of classes in $H^2(X)$ (resp. $H^2(X^\phi))$ that restrict trivially to a smooth fiber. Then
    $$
    \im r\cap \im r^\phi = r(H^2(X)^{0}) = r^\phi(H^2(X^\phi)^{0}).
    $$
\end{corollary}
\begin{proof}
    The image of the restriction map $r$ (resp. $r^\phi$) coincides with the kernel of $d_2$ (resp. $d_2^\phi$). By Proposition \ref{proposition d_2}, a class $\xi$ lies in the kernel of both $d_2$ and $d_2^\phi$ if and only if $\lambda_\xi = 0$, i.e., the restriction of $\xi$ to a smooth fiber is trivial. The claim follows.
\end{proof}

\begin{proposition}
\label{prop b_2}
    Let $\pi\colon X\to B$ be a Lagrangian fibration on an irreducible hyperk\"ahler manifold and $X^\phi$ its Shafarevich--Tate twist. Then either $b_2(X^\phi) = b_2(X)$ or $b_2(X^\phi) = b_2(X)-1$. The first case occurs if and only if there is a class $h\in H^2(X)$ whose restriction to a smooth fiber is non-trivial.
\end{proposition}
\begin{proof}
As will be explained below, it follows from the Leray spectral sequence for $\Q_{X^\phi}$ that
    \begin{equation}
    \label{eq b2}
    b_2(X^\phi) = b_2(B) + \dim H^1(R^1\pi_*\Q) + \rk r^\phi.
    \end{equation}
Indeed, 
$$
h^2(X^\phi) = E_\infty^{2,0} + E_\infty^{1,1} + E_\infty^{0,2}.
$$
The vector space $H^0(R^1\pi_*\Q)$ vanishes by Proposition \ref{proposition H1 vanishes}, hence the map $H^2(B)\to H^2(X^\phi)$ is injective and $E_\infty^{2,0} = E_2^{2,0} = H^2(B, \Q)$. Moreover, the second differential $d_2\colon H^1(R^1\pi_*\Q)\to H^3(B,\Q)$ is zero because $H^3(B,\Q) = 0$ \cite{shen2022topology}. Thus $E_\infty^{1,1} = E_2^{1,1} = H^1(R^1\pi_*\Q)$. Finally $E^{0,2}_\infty = \im r_\phi$. The formula (\ref{eq b2}) follows. Applying this formula to $X^\phi$ and $X$, we obtain that
$$
b_2(X^\phi) - b_2(X) = \rk r^\phi - \rk r.
$$

Since $X$ is hyperk\"ahler, the subspace $H^2(X)^{0}$ has codimension $1$ in $H^2(X)$, hence 
$$
    \rk r = \dim r(H^2(X)^{0}) + 1 = \dim r^\phi (H^2(X^\phi)^{0}) + 1,
$$
where the last equality holds by Corollary \ref{corollary differentials coincide} if $\overline\phi = 0$ and by Corollary \ref{corollary kernels equal} if $\overline\phi\ne 0$. If there is a class in $H^2(X^\phi)$ restricting non-trivially to a smooth fiber, then $H^2(X^\phi)^{0}$ is a hyperplane in $H^2(X^\phi)$ (Proposition \ref{proposition image of restriction map}) and
    $$
    \rk r^\phi = \dim r^\phi (H^2(X^\phi)^{0}) + 1 = \rk r.
    $$
In this case, $b_2(X) = b_2(X^\phi)$. Otherwise, $H^2(X^\phi)^{0} = H^2(X^\phi)$, hence
    $$
    \rk r^\phi = \dim r^\phi (H^2(X^\phi)^{0}) = \rk r - 1,
    $$
and 
    $$
    b_2(X) = b_2(X^\phi) + 1.
    $$
\end{proof}

\begin{proposition}
\label{prop classes non trivial on fibers}
    Let $\pi\colon X\to B$ be a Lagrangian fibration on an irreducible hyperk\"ahler manifold and $X^\phi$ its Shafarevich--Tate twist. Then there is a class $h\in H^2(X^\phi)$ that restricts non-trivially to a smooth fiber if and only if $\overline\phi \in H^2(\Gamma_\Q)$ is in the image of the boundary map $d_2\colon H^0(R^2\pi_*\Q)\to H^2(R^1\pi_*\Q)\simeq H^2(\Gamma_\Q)$.
\end{proposition}

\begin{proof}
    Suppose that there is a class $h\in H^2(X^\phi)$ that restricts non-trivially to a smooth fiber. Let $\overline h$ be its image in $H^0(B, R^2\pi_*\Q)$. By Proposition \ref{proposition d_2},
    $$
    0 = d_2^\phi(\overline h) = d_2(\overline h) + \lambda_h\overline\phi.
    $$
    Therefore, $\overline\phi = -d_2(\overline h)/\lambda_h$ is in the image of $d_2$.

    Conversely, suppose that $\overline\phi$ is in the image of $d_2$, i.e., there is a class $\xi\in H^0(R^2\pi_*\Q)$ such that $d_2\xi = \overline\phi$. Let $h_0$ be a K\"ahler class on $X$ with $\lambda_{h_0} = 1$ and $\overline{h_0}$ its image in $H^0(R^2\pi_*\Q)$. Consider the class
    $$
    \xi' := (1+\lambda_{\xi})\overline{h_0} - \xi.
    $$
    Then $\lambda_{\xi'} = 1$ and 
    $$
    d_2^\phi(\xi') = d_2(\xi') + \overline\phi = (1+\lambda_\xi)d_2(\overline{h_0}) - d_2(\xi) + \overline\phi = -d_2(\xi) + \overline\phi = 0.
    $$
    Here the first equality holds by Proposition \ref{proposition d_2}. Therefore, $\xi'$ lifts to a class in $H^2(X^\phi)$ which restricts non-trivially to smooth fibers.
\end{proof}

\subsubsection{Proof of Theorem \ref{theorem b2}.}\label{subsub proof of theorem c} Immediately follows from Propositions \ref{prop b_2} and \ref{prop classes non trivial on fibers}.\hfill$\qed$

\subsubsection{Proof of Theorem \ref{theorem fujiki class c}.}\label{subsub proof of theorem d}
    Suppose $X^\phi$ is of Fujiki class $\mathcal C$, i.e., there is a rational map $f\colon X^\phi\dashrightarrow Y$ to a K\"ahler manifold $Y$. Let $h\in H^2(X,\R)$ be the pullback of a K\"ahler form on $Y$. The restriction of $f$ to a general fiber $F$ of $\pi^\phi$ is birational onto its image, hence $h\restrict{F}$ is non-trivial. By Theorem \ref{theorem b2}, the twist satisfies $\overline\phi\in\im d_2$. \hfill $\qed$

\begin{remark}
Consider an abelian surface $A$ which is a product of elliptic curves $A=E\times F$. Let $K^n(A)$ be the generalized Kummer variety of $A$. It admits a Lagrangian fibration $\pi\colon K^n(A)\to \mathbb P^n$ whose general fiber is isomorphic to $F^n$. Let $p\colon S\to E$ be a primary Kodaira surface which is a principal torsor over $F$. This is a non-K\"ahler holomorphic symplectic surface. Consider its associated Bogomolov--Guan manifold $BG^n(S)$ \cite{guan1995examplesII, bogomolov1996guan}. It admits a Lagrangian fibration $\pi'\colon BG^n(S)\to \mathbb P^n$ whose general fiber is also isomorphic to $F^n$. Actually the non-K\"ahler holomorphic symplectic manifold $BG^n(S)$ is a Shafarevich--Tate twist of $K^n(A)$. As computed in \cite[Theorem 2]{guan1995examplesII}, $b_2(BG^n(S)) = 6$, which is exactly $b_2(K^n(A))-1$ in accordance with Theorem \ref{theorem b2}. Theorem \ref{theorem b2} also shows that the rank of the restriction map $H^2(BG^n(A))\to H^2(F^n)$ is zero. Details will be discussed in the upcoming paper \cite{abasheva_upcoming}.
\end{remark}




\bibliographystyle{alpha}
\bibliography{main.bib}

\noindent {\sc {Anna Abasheva} \\
New England Conservatory,\\
Department of Composition,\\
255 St Botolph St,\\
Boston, MA, USA.}\\
\\
{\em Formerly at}\\
{\sc Columbia University,\\
Department of Mathematics, \\
2990 Broadway, \\
New York, NY, USA}\\
{\tt aa4643(at)columbia(dot)edu}


  



\end{document}